\pdfoutput=1
\documentclass[11pt,a4paper]{article}
\usepackage{enumitem}

\usepackage{setspace}

\usepackage{amsmath}
\usepackage{amssymb}
\usepackage{amsfonts}
\usepackage{amsthm}

\usepackage{array}

\usepackage[numbers]{natbib}
\usepackage[page,toc]{appendix}

\usepackage[T1]{fontenc}
\usepackage{lmodern}

\usepackage{upquote}

\usepackage{listings}

\usepackage{algpseudocode}

\usepackage{tikz,pgfplots}
  \usetikzlibrary{fit,arrows.meta}
  \pgfplotsset{compat=1.3}

\usepackage{hyperref}
\makeatletter\g@addto@macro\UrlSpecials{\do\!{\discretionary{}{}{}}}\makeatother

\newcommand{\code}{\texttt}
\newcommand{\ExactpAdics}{\texttt{ExactpAdics}}
\newcommand{\ExactpAdicsII}{\texttt{ExactpAdics2}}

\newcommand{\FF}{\mathbb{F}}

\newcommand{\QQ}{\mathbb{Q}}
\newcommand{\RR}{\mathbb{R}}
\newcommand{\ZZ}{\mathbb{Z}}

\newcommand{\cF}{\mathcal{F}}

\newcommand{\cN}{\mathcal{N}}
\newcommand{\cO}{\mathcal{O}}

\newcommand{\fp}{\mathfrak{p}}

\newcommand{\Gal}{\operatorname{Gal}}

\newcommand{\val}{\operatorname{val}}
\newcommand{\xdiv}{\mathbin{\mathrm{div}}\penalty900}
\newcommand{\GL}{\operatorname{GL}}
\newcommand{\Res}{\operatorname{Res}}

\newcommand{\into}{\hookrightarrow}

\newcommand{\abs}[1]{\left\lvert#1\right\rvert}

\def\comment{}
\def\endcomment{}
\long\def\comment#1\endcomment{}

\theoremstyle{plain}
\newtheorem{lemma}{Lemma}[section]

\theoremstyle{definition}

\newtheorem{definition}[lemma]{Definition}
\newtheorem{example*}[lemma]{Example} 
\newenvironment{example}    
  {%
   \pushQED{\qed}\begin{example*}}
  {\popQED\end{example*}}
\theoremstyle{remark}
\newtheorem{remark}[lemma]{Remark}
\newcommand{\qedabove}{\vspace*{-1.3\baselineskip}\qedhere}
\newcommand{\qedhigher}{\vspace*{-0.6\baselineskip}\qedhere}

\makeatletter
\renewcommand*{\verbatim@font}{\ttfamily\fontseries{m}\selectfont}
\makeatother

\lstdefinelanguage{Magma}{
  morekeywords={end,function,intrinsic,procedure,for,while,repeat,until,do,in,if,else,elif,then,error,assert,require,when,where,is,print,printf,vprint,vprintf,time,declare,verbose,type,attributes,return,continue,break,delete,loop},
  morekeywords=[2]{eq,ne,le,lt,ge,gt,cmpeq,cmpne,not,notin,and,or,notsubset,subset,meet,join,diff,sdiff,assigned,eval},
  morekeywords=[3]{sub,ncl,func,proc,ideal,elt},
  morekeywords=[4]{AnyPadExact,StrPadExact,PadExactElt,FldPadExact,FldPadExactElt,RngUPol_FldPadExact,RngUPolElt_FldPadExact,RngMPol_FldPadExact,RngMPolElt_FldPadExact,SetCart_PadExactElt,Tup_PadExactElt,Val_PadExactElt,Val_FldPadElt,Val_RngUPolElt_FldPad,Val_RngMPolElt_FldPad,RngInt,RngIntElt,SetCart,Tup,List,FldNum,FldNumElt,FldRat,FldRatElt,FldPad,FldPadElt,Getter,BoolElt,SetCart_PadExact,Tup_PadExact},
  sensitive=true,
  morecomment=[l]{//},
  morecomment=[s]{/*}{*/},
  morestring=[b]",
}
\def\acknowledge{}
\newcommand{\acknowledgement}[1]{\def\acknowledge{\paragraph{Acknowledgements.} #1}}

\begin{document}

\title{\ExactpAdics: An exact representation \\ of \(p\)-adic numbers}
\author{Christopher Doris \\ University of Bristol \\ \texttt{christopher.doris@bristol.ac.uk}}
\date{April 2017}
\maketitle
\begin{abstract}
We describe two new packages \ExactpAdics{} and \ExactpAdicsII{} for the Magma computer algebra system for working with \(p\)-adic numbers exactly, in the sense that numbers are represented lazily to infinite \(p\)-adic precision. This has the benefits of increasing user-friendliness and speeding up some computations, as well as forcibly producing provable results. The two packages use different methods for lazy evaluation, which we describe and compare in detail. The intention is that this article will be of benefit to anyone wanting to implement similar functionality in other languages.
\end{abstract}
\acknowledgement{This work was partially supported by a grant from GCHQ.}


\section{Introduction}
\label{xp-sec-intro}

When dealing with completed fields, such as \(\RR\) or \(\QQ_p\), it is generally quite difficult to represent elements exactly. Instead, the commonest way to represent elements is by specifying them to some pre-determined precision, and then performing operations such as arithmetic to this precision also. This is the foundation of \emph{floating point arithmetic}. For example, one might represent the real number \(e\) by its approximation \(2.718281828\) to a precision of 10 real digits. We say such a representation is \emph{inexact} because several real numbers can have the same representation: \(e\), \(2.718281828\) and \(2.7182818281\) all have the same representation to 10 digits precision.

Such a representation is also usually \emph{zealous} meaning that when an operation is performed, such as multiplication, it is immediately computed to the required precision. For instance, computing \(e \times e\) will work to 10 digits precision and actually compute \(2.718281828 \times 2.718281828 = 7.389056096\). In fact, \(e \times e = 7.389056098\ldots\), demonstrating that precision errors can creep into the results, so that they are in fact less precise than the precision claims.

An often-suggested alternative to zealous arithmetic is \emph{lazy arithmetic}, wherein an operation does not produce an answer per-se, but a ``promise to produce an answer to a desired precision''. That is, calling \(e \times e\) would not produce the approximation \(2.718281828\), but would produce a function which, when called with an integer \(k\), returns an approximation to \(e \times e\) to \(k\) digits precision.

Such a function can be said to be an \emph{exact} representation of a real number, because no two distinct real numbers have the same representation: for a sufficiently large \emph{precision} \(k\), the representing functions will return different approximations.

These comments hold true for \(p\)-adic numbers too. For instance, an element of \(\QQ_p\) is generally represented in zealous, inexact arithmetic by its residue class in \(\QQ_p / p^k \ZZ_p\) for some \emph{absolute precision} \(k\): e.g. \(1 + 2^{10} \ZZ_2\) might represent \(1\), \(1 + 2^{10}\) or \(1 + 5 \times 2^{100}\).

There are numerous implementations of such \(p\)-adic arithmetic. FLINT \cite{flint} provides some low-level arithmetic with elements of \(\QQ_p\), univariate polynomials over \(\QQ_p\), and unramified extensions of \(\QQ_p\). Sage \cite{sage} and Magma \cite{magma} have more fully-featured implementations, including arbitrary finite extensions of \(\QQ_p\) and higher-level routines for tasks such as factoring.

Also of note is an implementation in Mathemagix \cite{mathemagix} of the so-called \emph{relaxed \(p\)-adic arithmetic}, which treats elements of \(\QQ_p\) like an infinite sequence of \(p\)-adic coefficients, somewhat like \(\FF_p((t))\), and represents them by a truncated sequence followed by a function to retrieve the next coefficient. This representation is therefore exact, because for different numbers, these streams of digits must eventually diverge. This has specific uses in \(p\)-adic recursion solving, and in principle is useful in general, but is somewhat more complicated to implement than the lazy arithmetic presented in this article, and as such is less fully featured.

A more in-depth description of different \(p\)-adic arithmetic systems is given by Caruso \cite{caruso}.

In this article, we present two new implementations of two different lazy, exact \(p\)-adic arithmetic systems. The implementations are written for the Magma computer algebra system \cite{magma} which, as mentioned above, already has a fully-featured implementation of zealous, inexact \(p\)-adic arithmetic. Our packages, called \ExactpAdics{} and \ExactpAdicsII{}, aim to use the inexact functionality already available as much as possible, in order to provide a more user-friendly wrapper. This allows for rapid addition of new features to the exact arithmetic as soon as they are available inexactly.

To the author's knowledge, these are the first highly-featured, general-purpose implementations of lazy \(p\)-adic arithmetic.

This article describes the rationale and the fundamental concepts behind the packages, but does not constitute a user manual. The user-manuals are available online at \url{https://!cjdoris.github.io/!ExactpAdics} and \url{https://!cjdoris.github.io/!ExactpAdics2}, and the packages may be downloaded from here also.

At the time of writing, we recommend the typical user to use the \ExactpAdicsII{} package (\S\ref{xp-sec-compare-ii}).

As an application, these packages has been used to implement the algorithm in \cite{conductor} to compute the 2-part of the conductor of a hyperelliptic curve of genus 2 defined over a number field. This implementation is available from \url{https://cjdoris.github.io/Genus2Conductor}. It uses such high-level \(p\)-adic routines as: computing the completion of a number field at a finite place (\S\ref{xp-sec-completions}); computing the factorization of a univariate polynomial (\S\ref{xp-sec-factorization}) and the fields defined by its factors; and Hensel-lifting roots of a system of multivariate equations (\S\ref{xp-sec-hensel-multiroot}).

As another application, these packages can optionally be used with the implementation of the algorithms described in \cite{galoisgroups} for computing the Galois group of a \(p\)-adic polynomial. This is available from \url{https://cjdoris.github.io/pAdicGaloisGroup}. With either package present the Galois group algorithm becomes provably correct, whereas otherwise with inexact \(p\)-adics there is no such guarantee. We also find that the algorithms run faster with exact \(p\)-adics, at least for reasonably high-degree inputs.

\acknowledge{}

\subsection{Terminology}
\label{xp-sec-terminology}

Suppose \(K\) is a \(p\)-adic field (a finite extension of \(\QQ_p\)), with ring of integers \(\cO=\cO_K\) and uniformizing element \(\pi=\pi_K\). The \(\pi\)-adic valuation is denoted \(\val=\val_K\) such that \(\val(\pi)=1\).

When we refer to an \emph{inexact} (representation of a) \(p\)-adic number \(x \in K\), we mean a conjugacy class \(x + \pi^k \cO\). We refer to \(k\) as the \emph{absolute precision} of (the representation of) the number.

Equivalently, it may be represented as \(\pi^v (y + \pi^r \cO)\) where \(y \in \cO\) and \(r \geq 0\). We refer to \(v\) as the \emph{weak valuation} of \(x\); it is a lower bound on the true valuation of \(x\). We refer to \(r\) as the \emph{relative precision}; it bounds the number of non-zero \(\pi\)-adic digits of \(x\) known. Note that \(v+r=k\).

We say that \(x\) is \emph{weakly zero} if \(y \in \pi^r \cO\), that is if the representation is of the form \(\pi^{v+r} \cO\). Note:
\begin{itemize}[noitemsep]
\item If \(x\) is not weakly zero, then it is not zero.
\item If \(r=0\) then \(x\) is weakly zero.
\end{itemize}

We typically enforce the following \emph{normalizing condition}: if \(r > 0\) then \(y \in \cO^\times\). Now note:
\begin{itemize}[noitemsep]
\item If \(x\) is not weakly zero, then its valuation is exactly \(v\).
\item \(x\) is weakly zero if and only if \(r=0\) (and if and only if \(k=v\)).
\end{itemize}

Magma's builtin \(p\)-adics (\verb|FldPad|, \verb|FldPadExact|, etc.) are inexact in this sense, and satisfy the normalizing condition. We note that \emph{prime} \(p\)-adic fields --- i.e. \(\QQ_p\) --- as opposed to their elements, can themselves naturally be represented exactly by the prime itself. Extensions of the form \(K(x)/(f(x))\) are usually represented inexactly via an inexact representation of the polynomial \(f(x) \in K[x]\); however we note that Magma does additionally have a builtin exact representation of extensions, represented by a map \(m : \ZZ \to K[x]\) such that \(m(k)\) is a defining polynomial to precision \(k\). We refer to this latter representation as \emph{semi-exact}, since the field is represented exactly but its elements are represented inexactly.

The residue class field \(\cO/\pi\cO\) is denoted \(\FF=\FF_K\), and \(\bar{x} \in \FF\) denotes the residue class of \(x \in \cO\).

A polynomial \(f(x) = \sum_{i=0}^d f_i x^i \in K[x]\) of degree \(d\) is \emph{Eisenstein} if \(\val(f_0)=1\), \(\val(f_i) \geq 1\) for \(1 \leq i < d\) and \(\val(f_d)=0\). It is irreducible, its roots have valuation \(\tfrac1d\), and so it defines a totally ramified extension \(K(x)/(f(x))\) of degree \(d\) such that \(x+(f(x))\) is a uniformizer.

A polynomial \(f(x) = \sum_{i=0}^d f_i x^i \in \cO[x]\) of degree \(d\) is \emph{inertial} if \(\val(f_d)=\val(f_0)=0\) and \(\bar{f}(x) = \sum_{i=0}^d \bar{f_i} x^i \in \FF[x]\) is irreducible over the residue class field \(\FF\). It is irreducible, the residue classes of its roots generate an extension of \(\FF\) of degree \(d\), and so it defines an unramified extension \(K(x)/(f(x))\).

\subsection{Comparison of zealous and lazy arithmetic}
\label{xp-sec-compare}

\subsubsection{Precision}

In zealous arithmetic, the user is generally required to choose a precision to work at in advance. Then all computations are performed to that precision, and it may happen that the precision chosen was not sufficient. In this case, the user will probably start the computation over with a higher precision. This process of manually increasing the precision of a computation can be burdensome for the user. In lazy arithmetic, such precision decisions are made automatically as far as possible.

\begin{example}
Here is a typical interactive Magma session, using its builtin lazy arithmetic:
\begin{lstlisting}
> // try to factorize at precision 10
> K := pAdicField(2, 10);
> R<x> := PolynomialRing(K);
> f := my_favourite_polynomial(R);
> Factorization(f);
error: ...
> // try to factorize at precision 20
> K := pAdicField(2, 20);
> R<x> := PolynomialRing(K);
> f := my_favourite_polynomial(R);
> Factorization(f);
error: ...
> // try to factorize at precision 40
> K := pAdicField(2, 40);
> R<x> := PolynomialRing(K);
> f := my_favourite_polynomial(R);
> Factorization(f);
[ <x^10 + ... >, ... ]
\end{lstlisting}

Using lazy arithmetic provided by our package, the equivalent session would be the following. Note that there is no explicit mention of precision.
\begin{lstlisting}
> K := ExactpAdicField(2);
> R<x> := PolynomialRing(K);
> f := my_favourite_polynomial(R);
> Factorization(f);
[ <x^10 + ... >, ... ]
\end{lstlisting}
\qedabove
\end{example}

In lazy arithmetic, each individual computation is performed to approximately the smallest precision it can be, and so precisions are very ``local'' in the computation. In zealous arithmetic, the precision is generally chosen once at the start of a computation, so each operation is performed to the same precision, and so precisions are more ``global''. If there is a single operation requiring a high ``global'' precision, this increases the precision that all other operations are performed to, which is a performance hit compared to lazy arithmetic.

\begin{example}
An example comes from the conductor algorithm mentioned in the introduction. One portion of this algorithm takes a polynomial \(f(x) \in \QQ_2[x]\), computes its factorization, chooses a factor \(g(x)\), computes the extension \(L/\QQ_2\) defined by \(g\), and then finds a root of \(g\) in \(L\). Usually, the precision required for the factorization far exceeds that of the root-finding; however, because the root-finding is over an extension \(L\), if it were to be done at the same high precision as the factorization, its run-time would often dominate.
\end{example}

\subsubsection{Correctness and provability}

When a \(p\)-adic number \(x \in K\) is represented inexactly as a class \(x + \pi^k \cO\), then it can be ambiguous whether it is really representing \(x\) or the class itself. For many operations, the distinction makes no difference; for example since \[(x+y)+\pi^k \cO = (x+\pi^k \cO) + (y + \pi^k \cO)\] then addition works the same in either interpretation. For other operations, Magma can produce potentially misleading answers; for example if \(x\) is represented as \(0 + \pi^k \cO\) then \verb|Valuation(x)| will return \(k\), when in fact all we really know is that \(\val(x) \geq k\).

\begin{definition}
Suppose \(F\) is a mathematical function and suppose \(\tilde F\) is a programmatic function intended to implement \(F\), so it takes as inputs representations of the inputs of \(F\) and returns as outputs representations of the outputs of \(F\). We say that \emph{\(\tilde F\) represents \(F\)} if for all possible inputs \(X\) to \(F\) and representations \(\tilde X\) of \(X\), that \(\tilde F(\tilde X)\) either does not return successfully or returns a representation of \(F(X)\).
\end{definition}

Hence if \(\tilde F\) represents \(F\), then its outputs depend only on the inputs being represented, and not on the representation of the inputs themselves. In the case of \(p\)-adic computation, this means that the output of \(\tilde F\) should not depend on the precision that its inputs were given to, and therefore is unambiguously a function of the \(p\)-adic value, and not its representation.

As already indicated, the \verb|Valuation| intrinsic in Magma does not represent the valuation function. Also equality is not represented, because it actually is equality of the representation: if \(x=1\) and \(y=1+2^{10}\) are both represented by \(1+2^{10} \ZZ_2\) then \(\verb|x eq y|\) will be true. In fact, it is not possible to determine that two \(p\)-adic numbers are equal when given to any finite precision, and it is only possible to tell that they are unequal if they are given to sufficiently large precision.

As another example, given a polynomial \(f(x)\) represented as \(\tilde f(x) = (1 + \pi^{10} \cO)x^2 + (0 + \pi^{10} \cO)\), the \verb|Roots| intrinsic in Magma will return a double root \(\tilde r = 0 + \pi^{10} \cO\) in \(K\). This is correct as a function of the representations themselves, since \(\tilde f(\tilde r) = 0 + \pi^{10} \cO\) represents 0, but if \(f(x) = x + 2^{11}\) then it is irreducible and therefore has no roots in \(K\). Similarly \verb|Factorization| and \verb|GCD| do not represent factorization and greatest common divisor of \(p\)-adic polynomials.

In our packages, if the name of an intrinsic function is the name of a mathematical function, then the intrinsic represents the function. For example, our \verb|Valuation| intrinsic (see Examples \ref{xp-ex-valuation} and \ref{xp-ex-valuation-ii}) will only return the true valuation of the given number; if the input is weakly zero, it may try to increase its precision, and could potentially do this forever (if the input is 0) or raise a precision error, but it is guaranteed that if it returns, its return value is correct.

In some cases, such as \verb|Roots| (\S\ref{xp-sec-roots}) and \verb|Factorization| (\S\ref{xp-sec-factorization}), the correctness of the output is forced by the fact that the outputs are given exactly. That is, if \verb|Roots| returns a root (exactly), then it by definition comes with a program to compute an approximation to the root to arbitrarily high precision, and therefore assuming the program is correct this is a proof that the root is correct. In the case of \verb|Roots|, it is Hensel's lemma which provides this proof.

The intrinsics which do not represent a function, and therefore depend on the representation, are given names which make this clear. The terms \verb|Weakly| and \verb|Definitely| are used to denote tests which can give false positives or false negatives; for example \verb|IsWeaklyZero| is true if its input appears to be zero up to some precision (but does not guarantee it is zero), and \verb|IsDefinitelyPrimitive| returns true if its input can be proven to be a primitive element (but if it returns false, this does not imply that its input is not primitive). Similarly the term \verb|Weak| denotes non-representing functions, so \verb|WeakValuation| returns the \(k\) in \(0 + \pi^k \cO_K\) and is therefore actually a lower bound on the true valuation; and \verb|WeakDegree| returns an upper bound on the degree of a polynomial, but which may be incorrect if its top coefficient is actually zero.

\subsubsection{Overheads}

The main down-sides of lazy arithmetic are the extra time and memory overheads introduced. In lazy arithmetic, \(p\)-adic values depend on other \(p\)-adic values, and all these dependencies need to be kept in memory for the duration of a computation. Each time an operation is performed, some dependency tracking and propagation needs to occur, which entails some processing time overhead.

This said, we find that these overheads do not usually dominate the run-time of lazy \(p\)-adic arithmetic unless one performs a large number of ordinarily very fast operations, such as basic arithmetic. If this is the case, then one can consider implementing the whole sequence of operations as a new atomic \(p\)-adic operation, which therefore now only contributes a single node to the dependency graph.

\subsection{Structure of this article}

The first three sections describe the \ExactpAdics{} package.

In \S\ref{xp-sec-core} we describe the core data types and functionality provided by the package, including a simplified description of the lazy evaluation of \(p\)-adic numbers.

In \S\ref{xp-sec-getter} we describe the lazy evaluation scheme actually employed by the package, which includes tracking dependencies between different \(p\)-adic values.

In \S\ref{xp-sec-strat} we describe ``precision strategies'', which are a way of programatically avoiding precision errors with minimal input from the user. This is not a core feature, but greatly improves user-friendliness.

\vspace{.5em}\noindent
Next we describe the \ExactpAdicsII{} package and compare.

In \S\ref{xp-sec-core-ii} we describe the core data types and functionality provided by the package, including a description of the lazy evaluation scheme.

In \S\ref{xp-sec-compare-ii} we compare the merits of the approaches taken by the two packages, including timings on some problems of interest.

\vspace{.5em}\noindent
The remaining sections describe additional features which either improve user-friendliness or provide more functionality. These features are mainly present in both packages.

In \S\ref{xp-sec-others} we describe additional structures which are not core functionality, namely multivariate polynomials and tuples.

In \S\ref{xp-sec-val} we describe our representation of valuations (defined in a generic sense) of \(p\)-adic objects, and the operations available on them.

Finally in \S\ref{xp-sec-features} we give an overview of additional features not covered elsewhere. This is largely to demonstrate that these packages are of practical use, since they include features such as root finding, factorization, residue classes and completions. We also provide some implementation notes.

\subsection{Pseudocode}

As the package is written in Magma, we shall use a simplification of the Magma language to demonstrate concepts\footnote{Specifically, we omit \code{;} and \code{end}, and imply code blocks through indentation. We also omit \code{\{documentation\}} blocks from intrinsics and \code{declare}.}. As this article may be useful to implement similar functionality in other languages, we summarise the syntax here.

Every variable has a \emph{type}. For example a ring of integers has type \verb|RngInt|, an integer has type \verb|RngIntElt|, a boolean (true or false) has type \verb|BoolElt| and an inexact \(p\)-adic field has type \verb|FldPad|. New types are defined as
\begin{lstlisting}
type NAME[ELT]: PARENT
\end{lstlisting}
where \verb|NAME| is the name of the type. The part in brackets is optional, but when given the type \verb|ELT| is also declared, and \verb|NAME| is actually a structure with elements of type \verb|ELT|. The part after the colon is optional, but when given the new type is a child of \verb|PARENT| in the type hierarchy, and in particular the new type inherits the attributes from \verb|PARENT|.

A type has attributes, which are named pieces of data attached to instances of the type. Attributes may be attached to a type by
\begin{lstlisting}
attributes TYPE: ATTR1, ATTR2, ...
\end{lstlisting}
where \verb|TYPE| is the name of the type, and \verb|ATTRn| are the names of the attributes.

Instances of the type are created like \verb|x := New(TYPE)|, and its attributes are accessed like \verb|x`ATTR|.

There are three types of functions in Magma: \verb|function|, \verb|procedure| and \verb|intrinsic|, all declared in a similar fashion, such as:
\begin{lstlisting}
function example(x, y : z := 0)
  return x + y + z
\end{lstlisting}

The difference between the three is that a function returns a value, and should not have any side-effects, a procedure does not return a value but can have side-effects (in particular an input may be passed by reference like \verb|~x| and it becomes modifiable), and an intrinsic is a function or procedure which forms the main user-interface. The \verb|z:=0| part is an optional parameter named \verb|z| whose default value is \verb|0|. Furthermore, intrinsics may have type declarations on its inputs and outputs, which allows overloading of intrinsics with the same name but different type signatures. For example:
\begin{lstlisting}
intrinsic ImportantExpression(
    x :: RngIntElt,
    y :: RngIntElt,
    z :: RngIntElt)
      -> FldRatElt
  return (x^2 + y^2) / z^3
\end{lstlisting}
is an intrinsic taking three integers and returning a rational.

Note that any pseudocode in this article is illustrative, and does not necessarily match the code in the implementation. The pseudocode is presented as simply as is possible to get the ideas across, whereas the real code will contain more checks and optimizations.

\section{\ExactpAdics: Core structures and elements}
\label{xp-sec-core}


\subsection{Abstract base types}

In \ExactpAdics{} we have an abstract base type \verb|StrPadExact| representing any kind of exact \(p\)-adic structure or set, and such a set has elements of type \verb|PadExactElt|:
\begin{lstlisting}
type StrPadExact[PadExactElt]
\end{lstlisting}

We shall later have sub-types representing the field of \(p\)-adic numbers (\S\ref{xp-sec-fldpadexact}), rings of polynomials over \(p\)-adic numbers (\S\ref{xp-sec-rngupol}), and more (\S\ref{xp-sec-others}).

Such a structure will always have an \verb|approximation| which is an analogous inexact structure:
\begin{lstlisting}
attributes StrPadExact: approximation
\end{lstlisting}

Elements always have a \verb|parent| structure to which they belong, as well as an \verb|approximation| and an \verb|update| function:
\begin{lstlisting}
attributes PadExactElt: parent, approximation, update
\end{lstlisting}

The \verb|approximation| is an element of the \verb|approximation| of the \verb|parent|, and so provides a finite-precision approximation to the element. The \verb|update| function provides the means to update the approximation arbitrarily precisely, and will be described later in this section and in \S\ref{xp-sec-getter}. Figure \ref{xp-fig-strpadexact} illustrates the relationships between these attributes.

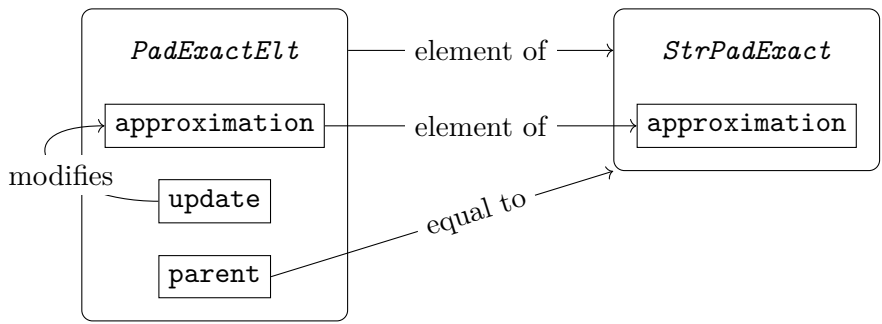
\begin{figure}
\centering
\begin{tikzpicture}[
  type/.style={font=\em\ttfamily},
  attr/.style={draw, font=\ttfamily},
  box/.style={draw, inner sep=0.3cm, rounded corners},
  label/.style={midway, fill=white},
]
\node(Et) at (0,-0) [type] {PadExactElt};
\node(Ea) at (0,-1) [attr] {approximation};
\node(Eu) at (0,-2) [attr] {update};
\node(Ep) at (0,-3) [attr] {parent};
\node(St) at (7,-0) [type] {StrPadExact};
\node(Sa) at (7,-1) [attr] {approximation};
\node(E)[box, fit={(Et) (Ea) (Eu) (Ep)}] {};
\node(S)[box, fit={(St) (Sa)}] {};
\draw[->](St -| E.east)--node[label]{element of}(Et -| S.west);
\draw[->](Ea)--node[label]{element of}(Sa);
\draw[->](Ep.east)--node[label,sloped,pos=0.6]{equal to}(S.south west);
\draw[->](Eu.west) to[out=180,in=180,looseness=3] node[label,pos=0.4]{modifies}(Ea.west);
\end{tikzpicture}
\caption[Illustration of \texttt{StrPadExact} and \texttt{PadExactElt}]{Illustration of the types \texttt{StrPadExact} and \texttt{PadExactElt}, their attributes and the relationships between them.}
\label{xp-fig-strpadexact}
\end{figure}

We provide some universal intrinsics to retrieve the parent of an element, its absolute precision and its weak valuation. The latter two are defined to be the absolute precision and weak valuation of the approximation, and therefore can change over time.
\begin{lstlisting}
intrinsic Parent(x :: PadExactElt) -> StrPadExact
  return x`parent

intrinsic AbsolutePrecision(x :: PadExactElt) -> .
  return AbsolutePrecision(x`approximation)

intrinsic WeakValuation(x :: PadExactElt) -> .
  return WeakValuation(x`approximation)
\end{lstlisting}

First we describe the representations of \(p\)-adic fields and rings of univariate polynomials.

\subsection{\texorpdfstring{\(p\)}{p}-adic fields}
\label{xp-sec-fldpadexact}

An exact \(p\)-adic field is represented by the type \verb|FldPadExact| (compare the name to the inexact \verb|FldPad| type in Magma) which derives from \verb|StrPadExact| (and so inherits its attributes) and has some additional attributes:
\begin{lstlisting}
type FldPadExact[FldPadExactElt]: StrPadExact
attributes FldPadExact: xtype, prime, defining_polynomial
\end{lstlisting}

The \verb|xtype| attribute takes one of the special enumerated values:
\begin{itemize}[noitemsep]
\item \verb|PRIME|: the field is \(\QQ_p\) for some \(p\), and the \verb|prime| attribute is \(p\).
\item \verb|INERT|: the field is an unramified extension of another exact \(p\)-adic field \(K\), and the \verb|defining_polynomial| attribute is an inertial polynomial \(f(x) \in K[x]\), defining the extension as \(K(x)/(f(x))\).
\item \verb|EISEN|: the field is a totally ramified extension of another exact \(p\)-adic field \(K\), and the \verb|defining_polynomial| is an Eisenstein polynomial \(f(x) \in K[x]\), defining the extension as \(K(x)/(f(x))\).
\end{itemize}

The \verb|approximation| field of an exact \(p\)-adic field is a corresponding inexact field, which in Magma has type \verb|FldPad|. For the \verb|PRIME| field \(\QQ_p\), this is simply \verb|pAdicField(p)|.

For extensions (\verb|INERT| or \verb|EISEN|) the situation is a little more complicated. In essence, we want the \verb|approximation| to be the extension defined by an approximation of the \verb|defining_polynomial|. The problem is that later we may want a more precise approximation, and so we have two choices:
\begin{itemize}
\item Replace the \verb|approximation| field with a more precise approximation whenever it is required. This means that any element of the field may have an \verb|approximation| lying in an older \verb|approximation| field, and so will need to be coerced into the latest \verb|approximation| field at some time.
\item Use Magma's built-in semi-exact representation of \(p\)-adic extensions (\S\ref{xp-sec-terminology}): \verb"ext<K | m>" where \verb|m| is a map taking an integer and returning an approximation of the \verb|defining_polynomial| to that precision.
\end{itemize}
We use the second choice because the explicit coercion between different \verb|approximation| fields in the first choice was found to add a performance hit. It also has the benefit that we can talk of \emph{the} \verb|approximation| field, since it does not change in time.

Elements of exact \(p\)-adic fields are represented by the type \verb|FldPadExactElt|. The meanings of its attributes are inherited from its parent type \verb|PadExactElt| but to be explicit:
\begin{itemize}[noitemsep]
\item \verb|parent| is the \verb|FldPadExact| field to which it belongs;
\item \verb|approximation| is an element of the \verb|approximation| field of its \verb|parent|, and is therefore a \verb|FldPadElt|;
\item \verb|update| is its update function, used to update the \verb|approximation|.
\end{itemize}

We define coercion so that \verb"K ! <init, mkupdate>" creates an element of \verb|K| whose initial approximation is \verb|init| and whose update function is \verb|mkupdate(x)| where \verb|x| is the element being created, thus allowing the update function to refer to \verb|x| itself:
\begin{lstlisting}
intrinsic IsCoercible(K :: FldPadExact, args :: Tup)
    -> FldPadExactElt
  x := New(FldPadExactElt)
  x`parent := K
  x`init := args[1]
  x`update := args[2](x)
  return true, x
\end{lstlisting}

We provide intrinsics to access basic information; intrinsics for inertia degree and ramification degree are defined similarly:
\begin{lstlisting}
intrinsic IsPrimeField(K :: FldPadExact) -> BoolElt
  return K`xtype eq PRIME

intrinsic DefiningPolynomial(K :: FldPadExact)
    -> RngUPolElt_FldPadExact
  if IsPrimeField(K) then
    error "not an extension"
  else
    return K`defining_polynomial

intrinsic BaseField(K :: FldPadExact) -> FldPadExact
  return BaseRing(DefiningPolynomial(K))

intrinsic Degree(K :: FldPadExact) -> RngIntElt
  return Degree(DefiningPolynomial(K))

intrinsic AbsoluteDegree(K :: FldPadExact) -> RngIntElt
  if IsPrimeField(K) then
    return 1
  else
    return Degree(K) * AbsoluteDegree(BaseField(K))
\end{lstlisting}

\subsection{Univariate polynomials}
\label{xp-sec-rngupol}

A univariate polynomial ring over a \(p\)-adic field is represented by the type \verb|RngUPol_FldPadExact| (analogous to the inexact type \verb|RngUPolElt[FldPad]| in Magma) which also derives from \verb|StrPadExact|:
\begin{lstlisting}
type RngUPol_FldPadExact[RngUPolElt_FldPadExact]
attributes RngUPol_FldPadExact: base_ring
\end{lstlisting}

Such a ring is defined by its \verb|base_ring|, an exact \(p\)-adic field (i.e. of type \verb|FldPadExact|).

The \verb|approximation| of such a ring must be the univariate \verb|PolynomialRing| of the \verb|approximation| of the \verb|base_ring| (i.e. of type \verb|RngUPol[FldPad]|).

\subsection{The update function}
\label{xp-sec-update}

The \verb|update| attribute of a \verb|PadExactElt| is a means to increase the precision of its \verb|approximation| to a given absolute precision. Therefore it is natural to define it as a procedure which takes as input an absolute precision \(k\), and whose side-effect is to replace the \verb|approximation| by one whose precision is at least \(k\). Using this definition will result in a working implementation of exact \(p\)-adics, but as we shall see in \S\ref{xp-sec-getter} it has some drawbacks and so in reality we use a slightly different definition. For now, however, it suffices to think of the update function in this way.

In the update function, instead of modifying the \verb|approximation| of an element directly, one should use the following intrinsic which first checks that the update is consistent with the pre-existing approximation and in reality may perform more checks:

\begin{lstlisting}
intrinsic Update(x :: FldPadExactElt, xx :: FldPadElt)
  assert IsWeaklyEqual(x`approximation, xx)
  x`approximation := xx
\end{lstlisting}

Instead of calling the \verb|update| function directly, we should use the following intrinsic which ensures it is only called when required. In fact, since the \verb|update| function is not a function at all then this intrinsic will actually have a different definition (\S\ref{xp-sec-update2}), but with the same effect.

\begin{lstlisting}
intrinsic IncreaseAbsolutePrecision(x :: PadExactElt, n)
  if not AbsolutePrecision(x) ge n then
    x`update(n)
\end{lstlisting}

In practice, we don't usually just want to increase the precision of an element, but we want the approximation itself. Hence we also make available an intrinsic \verb|Approximation| to retrieve an approximation to an element to a certain absolute precision. It simply has to increase the absolute precision of the element, then return its approximation, perhaps with its precision decreased to the desired value.

\begin{lstlisting}
intrinsic Approximation(x :: FldPadExactElt, n)
  IncreaseAbsolutePrecision(x, n)
  return ChangeAbsolutePrecision(x`approximation, n)
\end{lstlisting}

To increase the precision of an extension field, we just need to increase the precision of its \verb|defining_polynomial| correspondingly. This ensures that the next time the semi-exact \verb|approximation| field retrieves a defining polynomial, it will already be available to the given precision. There is nothing to be done for \verb|PRIME| fields, since the prime is already represented exactly.

\begin{lstlisting}
intrinsic IncreasePrecision(K :: FldPadExact, n)
  if not IsPrimeField(K) then
    IncreaseAbsolutePrecision(K`defining_polynomial, n)
\end{lstlisting}

The precision of a polynomial ring is the precision of its base ring, so to increase one we just have to increase the other:

\begin{lstlisting}
intrinsic IncreasePrecision(R :: RngUPol_FldPadExact, n)
  IncreasePrecision(R`base_ring, n)
\end{lstlisting}

\subsection{Examples}

\begin{example}
\label{xp-ex-add}
Here is a definition of binary addition on two \(p\)-adic numbers. The initial approximation is simply the sum of the approximations of the inputs. The update function retrieves approximations to the inputs to the required precision, adds them, and sets this as the new approximation for the sum.
\begin{lstlisting}
intrinsic '+' (x :: FldPadExactElt, y :: FldPadExactElt)
    -> FldPadExactElt
  init := x`approximation + y`approximation
  mkupdate := function (z)
    return procedure (n)
      Update(z, Approximation(x, n) + Approximation(y, n))
  return Parent(x) ! <init, mkupdate>
\end{lstlisting}
This example is an over-simplification compared to the true implementation in the following ways:
\begin{itemize}
\item The update function is not quite as described. See \S\ref{xp-sec-getter}.
\item The initial approximation \verb|init| is computed to the current precision of the inputs, which may be overkill if they are both very precise. Instead, the implementation adds together approximations to ``first precision'', i.e.
\begin{verbatim}
init := ChangeAbsolutePrecision(x`approximation,
  Min(WeakValuation(x)+1, AbsolutePrecision(x)))
  + ChangeAbsolutePrecision(y`approximation,
  Min(WeakValuation(y)+1, AbsolutePrecision(y)))
\end{verbatim}
As an optimization, most functions will compute the initial approximation from the inputs to first precision if possible.
\item It should be checked that the inputs have the same \verb|parent| field, or can be coerced to a common field. \qedhere
\end{itemize}
\end{example}

\begin{example}
Here we give a definition of binary multiplication, which is very similar to addition. The main change is in computing the precision required in the approximations.
\begin{lstlisting}
intrinsic '*' (x :: FldPadExactElt, y :: FldPadExactElt)
    -> FldPadExactElt
  init := x`approximation * y`approximation
  mkupdate := function (z)
    return procedure (n)
      Update(z, Approximation(x, n - WeakValuation(y))
        * Approximation(y, n - WeakValuation(x)))
  return Parent(x) ! <init, mkupdate>;
\end{lstlisting}
\qedabove
\end{example}

\section{\ExactpAdics: Dependency tracking}
\label{xp-sec-getter}


\subsection{Motivation}
\label{xp-sec-getter-motivation}

So far we have described a simple scheme for implementing exact \(p\)-adics, but it has drawbacks.

\begin{example}
\label{xp-ex-getter}
Suppose we are given elements \(a, b \in \QQ_p\), and compute \(c = a^3 + a^2 b + a b^2 + b^3\), and wish to increase the absolute precision of \(c\) to 100.

We therefore require each of the summands \(a^3\), \(a^2 b\), \(a b^2\), \(b^3\) to absolute precision 100. Now suppose that \(\val(a)=10\) and \(\val(b)=0\), so in fact we require the summands to relative precisions 70, 80, 90, 100 respectively. Hence we require \(a\) and \(b\) to these same relative precisions.

Therefore, if we increase the precision of each summand in turn to its required value, then we will be updating \(a\) first to relative precision 70, then 80, then 90 --- i.e. absolute precisions 80, 90 and 100 --- which is \(3\) separate updates. If updating \(a\) is an expensive operation, then this could become a performance issue.

Clearly, the right thing to do in this situation is to observe that we only need to update \(a\) once to absolute precision 100. With the current description of the \verb|update| function, this is not possible.
\end{example}

Our solution is to split updates into two steps: the first step identifies which other updates are required to occur first, we call these \emph{dependencies}; the second step actually performs the update. With this explicit separation, we can find all of the dependencies of a calculation before satisfying any of them, allowing us to remove any redundancy as in the above example.

In the example, \(c\) has 4 dependencies, namely the 4 summands. Each of these summands in turn depends on one or both of \(a\) and \(b\). There is redundancy in these dependencies because \(a\) and \(b\) each appear three times, and therefore could be merged.

\subsection{Getters}

We encapsulate these ideas into a new type\footnote{In the package it is actually called \texttt{ExactpAdics\_Gettr}}:

\begin{lstlisting}
type Getter
attributes Getter: state, get_dependencies, get_value
\end{lstlisting}

A \verb|Getter| represents an exact \(p\)-adic computation with dependencies. The \verb|state| attribute is some getter-specific state which is passed by reference (and hence is modifiable) into the other functions.

The \verb|get_dependencies| attribute is a \verb|procedure(~state, ~deps)| which assigns to \verb|deps| a list of dependencies. A dependency is a pair \verb|<x,n>| where \verb|x| is some \(p\)-adic value (i.e. a \verb|PadExactElt|, such as a \(p\)-adic number or polynomial) and \verb|n| is an absolute precision. Such a dependency should be interpreted as the getter saying ``I can't compute my value until these values are to these absolute precisions.'' We say a dependency is \emph{satisfied} if the absolute precision of \verb|x| is at least \verb|n|.

The \verb|get_value| attribute is a \verb|procedure(~state, ~value)| which, assuming that the dependencies previously reported are all satisfied, either assigns something to \verb|value| or doesn't. If it does, then this is interpreted as the value of the computation. If it doesn't, then this is interpreted as the getter having more dependencies, and so \verb|get_dependencies| needs to be called again.

Evaluating a getter means getting the value from the \verb|get_value| procedure. Of course, this requires satisfying the dependencies reported by \verb|get_dependencies| first, and leads to a recursive dependency satisfaction algorithm which we describe shortly.

\subsection{Update function}
\label{xp-sec-update2}

With getters defined, we may now define precisely what an update function is: it is a function taking as input an absolute precision \verb|n| and returning a \verb|Getter|. This getter, when evaluated, will have the side-effect of increasing the absolute precision of the element to \verb|n|. The value of the getter is ignored.

With this definition, \verb|IncreaseAbsolutePrecision| would actually be defined like so:

\begin{lstlisting}
intrinsic IncreaseAbsolutePrecision(x :: PadExactElt, n)
  if not AbsolutePrecision(x) ge n then
    ignored := Evaluate(x`update(n))
\end{lstlisting}

Now to increase the absolute precision of a value, we just need to know how to evaluate a getter.

\subsection{Evaluating getters}
\label{xp-sec-getter-eval}

To evaluate a getter requires conceptually three steps: first we retrieve its dependencies from \verb|get_dependencies|, then we satisfy those dependencies, then we retrieve the value via \verb|get_value|. If \verb|get_value| did not return a value, then we will need to repeat these steps.

Each of the dependencies is a pair \verb|<x, n>| of a \(p\)-adic element and an absolute precision. Calling \verb|x`update(n)| returns a getter which, on evaluation, increases the absolute precision of \verb|x| to \verb|n|, which we require. Hence we have reduced the problem of evaluating the original getter to the problem of evaluating these dependent getters. Recursing, we traverse the tree of dependencies all the way to its leaves. Figure \ref{xp-fig-deps} illustrates this dependency tree for the motivating example (assuming \(a\) and \(b\) themselves have no dependencies).

\begin{figure}
\centering
\begin{tikzpicture}
\node(c) at (6,0) {\(a^3 + a^2 b + a b^2 + b^3\), 100};
\node(a3) at (1,-1) {\(a^3\), 100};
\node(a2b) at (4,-1) {\(a^2 b\), 100};
\node(ab2) at (8,-1) {\(a b^2\), 100};
\node(b3) at (11,-1) {\(b^3\), 100};
\node(a2) at (3,-2) {\(a^2\), 100};
\node(b) at (5,-2) {\(b\), 80};
\node(a) at (7,-2) {\(a\), 100};
\node(b2) at (9,-2) {\(b^2\), 90};
\node(a_a3) at (1,-2) {\(a\), 80};
\node(b_b3) at (11,-2) {\(b\), 100};
\node(a_a2) at (3,-3) {\(a\), 90};
\node(b_b2) at (9,-3) {\(b\), 90};
\draw(c)--(a3);
\draw(c)--(a2b);
\draw(c)--(ab2);
\draw(c)--(b3);
\draw(a2b)--(a2);
\draw(a2b)--(b);
\draw(ab2)--(a);
\draw(ab2)--(b2);
\draw(a3)--(a_a3);
\draw(a2)--(a_a2);
\draw(b3)--(b_b3);
\draw(b2)--(b_b2);
\end{tikzpicture}
\caption[Tree of dependencies]{The tree of dependencies of the motivating example (Example \ref{xp-ex-getter}).}
\label{xp-fig-deps}
\end{figure}
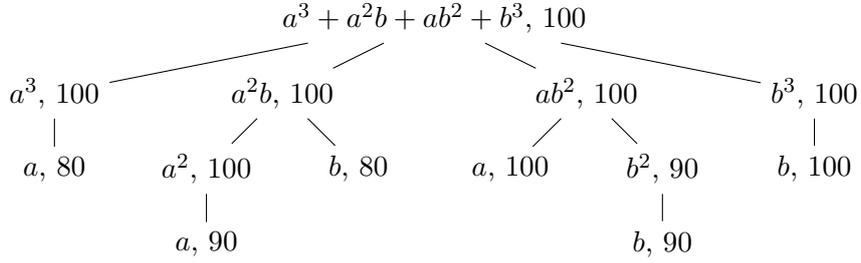

To avoid duplicated work, we want to combine all nodes for the same value together, taking the maximum of their absolute values, resulting in a directed acyclic graph such as in Figure \ref{xp-fig-deps2}.

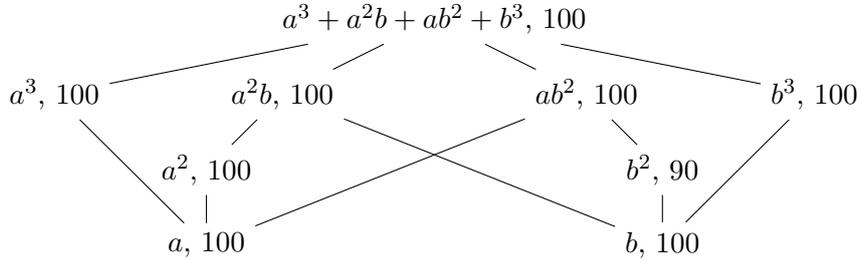
\begin{figure}
\centering
\begin{tikzpicture}
\node(c) at (6,0) {\(a^3 + a^2 b + a b^2 + b^3\), 100};
\node(a3) at (1,-1) {\(a^3\), 100};
\node(a2b) at (4,-1) {\(a^2 b\), 100};
\node(ab2) at (8,-1) {\(a b^2\), 100};
\node(b3) at (11,-1) {\(b^3\), 100};
\node(a2) at (3,-2) {\(a^2\), 100};
\node(b) at (9,-3) {\(b\), 100};
\node(a) at (3,-3) {\(a\), 100};
\node(b2) at (9,-2) {\(b^2\), 90};
\draw(c)--(a3);
\draw(c)--(a2b);
\draw(c)--(ab2);
\draw(c)--(b3);
\draw(a2b)--(a2);
\draw(a2b)--(b);
\draw(ab2)--(a);
\draw(ab2)--(b2);
\draw(a3)--(a);
\draw(a2)--(a);
\draw(b3)--(b);
\draw(b2)--(b);
\end{tikzpicture}
\caption[Merged graph of dependencies]{The merged graph of dependencies of the motivating example (Example \ref{xp-ex-getter}).}
\label{xp-fig-deps2}
\end{figure}

A sink in this graph is precisely a getter with no dependencies. Therefore it may be evaluated and removed from the graph. Repeating, we will eventually reach the source, which can also be evaluated, and we have succeeded.

In practice, we do not represent these dependencies as a tree at all, but we take advantage of a simple fact: if the value \verb|x| was created before value \verb|y|, then \verb|x| cannot possibly depend on \verb|y|. We therefore keep track of the order of creation of elements by giving them a new attribute
\begin{lstlisting}
attributes PadExactElt: id
\end{lstlisting}
to which we assign the value of a global counter when the element is created. We now represent the nodes of the graph simply as an associative array, where the node \verb|<x,n>| is the value at the index \verb|x`id|. Adding a new dependency into the tree is a matter of checking if there is already a dependency for this \verb|x|; if so, then we should combine the old and new absolute precision in the array; if not, then we add the new node into the array. Traversing the tree in dependency order is now a matter of runnning through the indices of the array in sorted order.

We now present a version of this algorithm. The first procedure \verb|add_dependencies| takes a list of dependency pairs \verb|<x,n>| and recursively adds them and all their own dependencies into the array.

\begin{lstlisting}
procedure add_dependencies(~array, todo_list)
  while #todo_list gt 0 do
    // pop an item from the todo list
    x, n := Pop(~todo_list)
    // if this is a new dependency, or the target
    // precision is greater than the existing one,
    // and it is not already satisfied, then replace
    // it and compute more dependencies to put in
    // the todo list
    if (x`id notin array or not n le array[x`id][2])
    and not n le AbsolutePrecision(x)
    then
      getter := x`update(n)
      array[x`id] := <x, n, getter>
      getter`get_dependencies(~getter`state, ~deps)
      for dep in deps do
        Append(~todo_list, dep)
\end{lstlisting}

The procedure \verb|satisfy_dependencies| takes a list of dependency pairs and satisfies them all, first by calling \verb|add_dependencies| to make an array of all dependencies, and then by running through the dependencies in order and trying to satisfy them. Here is one possible implementation:

\begin{lstlisting}
procedure satisfy_dependencies(deps)
  // initially compute all dependencies
  array := AssociativeArray()
  add_dependencies(~array, deps)
  // keep trying until the graph is empty
  while #array gt 0 do
    // traverse the nodes in order
    for i in Sort(Keys(array)) do
      // do the update
      getter := array[i][3]
      getter`get_value(~getter`state, ~value)
      if assigned value then
        // success: remove the entry from the array
        delete array[i]
      else
        // failure: get more dependencies and start over
        getter`get_dependencies(~getter`state, ~deps)
        add_dependencies(~array, deps)
        break
\end{lstlisting}

The implementation in the \ExactpAdics{} package behaves a little differently: if a particular node fails, then instead of immediately jumping back to the bottom of the tree, we continue traversing it to the top, but skipping over nodes which now have unsatisfied depdendencies. This is made possible by changing \verb|add_dependencies| to explicitly track the children and parents of each node in the array, and altering \verb|satisfy_dependencies| to:

\begin{lstlisting}
procedure satisfy_dependencies(deps)
  // initially compute all dependencies
  array := AssociativeArray()
  add_dependencies(~array, deps)
  // keep trying until the graph is empty
  while #array gt 0 do
    // traverse the nodes in order
    for i in Sort(Keys(array)) do
      item := array[i]
      if item has no children then
        // do the update
        getter := item[3]
        getter`get_value(~getter`state, ~value)
        if assigned value then
          // success: remove the entry from the array
          for each parent of item do
            remove item as a child of parent
          delete array[i]
        else
          // failure: get more dependencies
          getter`get_dependencies(~getter`state, ~deps)
          add_dependencies(~array, deps)
\end{lstlisting}

Observe the main differences are that we now need to check a node has no children before processing it; when a node succeeds, we now need to remove it from the list of children of each of its parents; and when it fails, we no longer \verb|break| out from looping over nodes.

Which of these two routines is better is arguable. The former routine may suffer from updating an early element of the graph, and then discovering later that the same element needs to be updated again, whereas the latter avoids this problem by performing as many updates on the graph as possible before starting over again. On the other hand, the latter routine may suffer from traversing the whole graph needlessely if an early node failed and is a child of everything else.

In practice, on problems of interest, we found that the latter routine usually performed better. That is, it was sometimes significantly faster and was rarely significantly slower, which is why the latter is used in the current implementation.

We also now define the intrinsic which evaluates a getter:

\begin{lstlisting}
intrinsic Evaluate(g :: Getter)
  loop
    g`get_dependencies(~g`state, ~deps)
    satisfy_dependencies(~deps)
    g`get_value(~g`state, ~value)
    if assigned value then
      return value
\end{lstlisting}

\subsection{Lazy computations}

Now that we have a way of representing computations with dependencies, we now define some ways of combining and modifying them to produce more complex computations, with dependency tracking still built-in.

For instance, \verb|Compose| takes as input a getter \verb|g| and a function \verb|f|, and returns a getter \verb|h| such that \verb|Evaluate(h)| \verb|=| \verb|f(Evaluate(g))|.

Similarly, \verb|ComposeProcedure| takes as input a getter \verb|g| and a procedure \verb|f|, and returns a getter \verb|h| such that \verb|Evaluate(h)| has the same side-effects as calling \verb|f(Evaluate(g))|.

Similarly, \verb|ComposeGetter| takes as input a getter \verb|g| and a function \verb|f| returning a getter, and returns a getter \verb|h| such that \verb|Evaluate(h)| \verb|=| \verb|Evaluate(f(Evaluate(g))|.

For added convenience, these compose functions can take a sequence of getters instead a single getter. In this case, the arity of the function \verb|f| must equal the length of the sequence. For example \verb|Evaluate(Compose([g1, g2], f))| \verb|=| \verb|Evaluate(f(Evaluate(g1), Evaluate(g2)))|.

The intrinsic \verb|Flatten| takes as input a sequence of getters, and returns the getter whose value is the sequence of values of the input getters.

There are also intrinsics for defining null getters, which do nothing, and for defining getters directly in terms of the \verb|get_value| and \verb|get_dependencies| functions. It is also possible to define getters direcly whose dependencies are themselves getters, instead of \verb|<x,n>| element-precision pairs.

The package also provides ``lazy'' versions of some intrinsics, which by convention are given the suffix \verb|_Lazy|, which returns a getter which when evaluated has the same side-effects and return value as the non-lazy version.

For example the following intrinsic returns a getter \verb|g| such that \verb|Evaluate(g)| has the same side-effects as calling \verb|IncreaseAbsolutePrecision(x, n)| directly:

\begin{lstlisting}
intrinsic IncreaseAbsolutePrecision_Lazy
  (x :: PadExactElt, n) -> Gettr
  if AbsolutePrecision(x) ge n then
    return NullGetter()
  else
    return x`update(n diff AbsolutePrecision(x))
\end{lstlisting}

Indeed, one could now define the non-lazy version as
\begin{lstlisting}
intrinsic IncreaseAbsolutePrecision(x :: PadExactElt, n)
  ignored := Evaluate(IncreaseAbsolutePrecision_Lazy(x,n))
\end{lstlisting}

Similarly, the intrinsic \verb|Approximation_Lazy(x, n)| returns a getter whose value is an approximation of \verb|x| to absolute precision \verb|n|.

\begin{example}
\label{xp-ex-add2}
We can now present an implementation of binary addition using these tools. Compare this with the earlier version (Example \ref{xp-ex-add}), presented in terms of the simplified representation where the update function was simply a procedure; now it is a getter, built using \verb|ComposeProcedure| and \verb|Approximation_Lazy| out of simpler getters.

\begin{lstlisting}
intrinsic '+' (x :: FldPadExactElt, y :: FldPadExactElt)
    -> FldPadExactElt
  init := x`approximation + y`approximation
  mkupdate := function (z)
    return function (n)
      return ComposeProcedure(
        // lazily computes approximations to x and
        // y to precision n
        [ Approximation_Lazy(x, n)
        , Approximation_Lazy(y, n) ],
        // uses the approximations xx and yy to update
        // the value of z
        procedure (xx, yy)
          Update(z, xx + yy)
      )
  return Parent(x) ! <init, mkupdate>
\end{lstlisting}
\qedabove
\end{example}

Most update functions in the package are defined in a similar fashion: firstly they lazily compute approximations to their inputs, and then they use these to update the value. The exceptions to this are mainly when the precision required of the input is not known immediately, and therefore some iteration is required; in such a circumstance, the getter returned by the update function usually needs to be defined directly in terms of its \verb|get_dependencies| and \verb|get_value| procedures.

\section{\ExactpAdics: Precision strategies}
\label{xp-sec-strat}


\subsection{Motivating example}

Suppose we want to compute the valuation of a \(p\)-adic number \(x\). If the number is not currently weakly zero, then this is straighforward: the valuation is the weak valuation. Otherwise, we can't immediately deduce the valuation.

Therefore, we might try increasing the absolute precision of \(x\). If the number is now not weakly zero, then we are done. Otherwise, we might increase the absolute precision further, and repeat this process for some time. At some point, if we don't discover the answer, we might give up.

How do we choose the amount to increase the absolute precision by? How long do we go for before giving up? A simple answer might be to keep doubling the precision forever until we succeed, but this would never terminate if \(x=0\).

On the other hand, the user may want the process to definitely terminate after some amount of effort, and therefore give up after the precision has reached some limit. Or, knowing more about the inputs, it may be more appropriate to increase the precision linearly instead of exponentially, for example. We abstract away such decisions into a precision strategy.

\subsection{Definition}

A \emph{precision strategy} is a strictly increasing sequence of non-negative integers. The sequence may be finite or infinite in length.

\subsection{Representation}

How such a sequence is represented is not too important. In the \ExactpAdics{} package, a precision strategy is represented as one of the following:
\begin{itemize}
\item A single non-negative integer \(n\), which is a strategy of length 1: \((n)\).
\item A list of strategies, which is the contatenation of those strategies.
\item A function \(m\) which takes an integer and either returns true and a larger value, or returns false. It represents a sequence of integers as follows: let \(n_0\) be the previous value in the strategy; for \(i=0,1,\ldots\), if \(m(n_i)\) is false, then terminate the sequence, otherwise it is true and also returns \(n_{i+1}\), the next element of the sequence.
\item A string, which is interpreted as the global strategy with that name (see below).
\item A tuple \verb|<"limit",n>| which limits the remaining strategy to \(n\); that is, it will terminate when the strategy reaches \(n\). More precisely, the first time the strategy outputs a number \(m \geq n\), it instead outputs \(n\) itself and then terminates.
\item A tuple \verb|<"exp",e>| which is equivalent to the function taking \(n\) to \(\lceil n^e \rceil\). It therefore represents an infinite sequence which grows exponentially.
\item A tuple \verb|<"random">| which randomises the remaining strategy as follows: if at time \(i\) the previous value was \(n_i\) and the next value will be \(n_{i+1} > n_i\), then we replace the next value with a uniform random element in \(\{n_i+1,\ldots,n_{i+1}\}\). This aims to dampen any potential issues arising from forcing precisions to come from a small set of values, such as powers of 2.
\end{itemize}

We provide the user with a global array of named strategies, with procedures to define and retrieve strategies with a particular name. Currently we define three named strategies by default:
\begin{itemize}
\item \verb|"defaultLimit"|: \verb|<"limit", 100>|, not a precision strategy in itself, but can be mixed in to other strategies to limit them.
\item \verb|"unlimitedDefault"|: \verb|[1, <"randomize">, <"exp", 2>]|, starts at 1 and keeps doubling forever.
\item \verb|"default"|: \verb|["defaultLimit", "unlimitedDefault"]|, the same as the previous strategy, but with the default limit applied.
\end{itemize}

\subsection{Usage and conventions}

Any function which makes a non-canonical decision about how to control the precision of its inputs should take one or more precision strategies as parameters to make these decision.

By convention, these parameters all have the word \verb|Strategy| in their name, to make their purpose clear. Where there are any precision strategies parameters, there will be one with the name precisely \verb|Strategy|. Its value is used as the default value for the others. Its default value is the string \verb|"default"|, which refers to the global strategy with this name. For example:

\begin{lstlisting}
intrinsic DoSomething( : Strategy   := "default",
                         Strategy1  := Strategy,
                         Strategy2  := Strategy,
                         Strategy2b := Strategy2)
  ...
\end{lstlisting}

This means that for the typical user, it suffices to set a global strategy with the name \verb|"default"| and then forget about precision strategies. If it then turns out that some computations are raising precision errors, the user can consider altering the precision strategy.

Using a precision strategy whose maximum value is 100 is functionally very similar to using inexact \(p\)-adics to precision 100. The difference is that in the inexact case, all computations are done to this precision, whereas in the exact case, 100 is the worst case: if a computation can be done with less precision, then it will be.

\subsection{Baseline precision}

A function which takes a precision strategy parameter is free to use it in any fashion. It is intended, of course, that it will be interpreted as a sequence of precisions to try computations at, but there are different kinds of precision:
\begin{itemize}
\item Absolute precision: the integer \(k\) such that we know the value modulo \(\pi^k\); that is, the approximation is of the form \(x + \pi^k \cO\).
\item Relative precision: the absolute precision minus the weak valuation; that is, the integer \(r\) such that the approximation is of the form \(\pi^v(x + \pi^r \cO)\).
\end{itemize}

If we interpret the entries of a precision strategy as absolute precisions, then it may be that the \(p\)-adic number actually has a large negative valuation, and so computing it to positive absolute precision is overkill. In a sense, the absolute precision is relative to the valuation 0, and 0 is an arbitrary choice.

If we interpret them as relative precisions, then we lose repeatability because the base-line for the relativeness can move: the weak valuation may increase in time. To demonstrate the issue, suppose we are computing the valuation of \(x=0\), initially known to absolute precision 10, and the precision strategy goes up to 100. Interpreting the strategy as relative precisions, we would increase the absolute precision of \(x\) to 110. If we try to compute the valuation again, then we will increase the absolute precision again to 210. There is the potential to keep increasing the absolute precision of \(x\) indefinitely by repeatedly trying to compute its valuation.

We introduce a new kind of precision:
\begin{itemize}
\item Baseline precision: the absolute precision minus the ``baseline valuation''.
\end{itemize}

The baseline valuation is any fixed valuation attached to the value. Hence it may depend on the value, but does not change over time. By default, the baseline valuation is set to the weak valuation of the value when it is initially created. If the baseline valuation is set to 0, then we recover the absolute precision.

If we now interpret entries of the precision strategy as baseline precisions, then we avoid the two problems described above.

\begin{example}
\label{xp-ex-valuation}
Hence, a reasonable implementation of a function to compute the valuation is
\begin{lstlisting}
intrinsic Valuation(x :: FldPadExactElt : Strategy:="default")
  for n in Strategy do
    IncreaseAbsolutePrecision(x, BaselineValuation(x) + n)
    if not IsWeaklyZero(x) then
      return WeakValuation(x)
  error "precision error"
\end{lstlisting}
\qedabove
\end{example}

\section{\ExactpAdicsII: Core structures and elements}
\label{xp-sec-core-ii}

\subsection{Overview}

Recall that in the \ExactpAdics{} package, our updates are performed in terms of absolute precisions: an update function receives an absolute precision, and updates the approximation accordingly. In order to do this, the update function must compute the absolute precisions required of its dependencies, which are fed into the dependency-tracking framework, and recursively the update functions of the dependencies themselves must compute the absolute precisions required of their dependencies.

In the \ExactpAdicsII{} package, we simplify this procedure by introducing a proxy for absolute precision. This proxy is a single positive integer \(n\) which we refer to as the \emph{epoch}. At any given time, a \(p\)-adic object has a \emph{current epoch} meaning that its current approximation is associated to that epoch. The precision of the current approximation must increase with the epoch.

Importantly, \emph{by definition} the approximation of a \(p\)-adic object at epoch \(n\) depends only on the approximations of its dependencies at epoch \(n\). Hence a \(p\)-adic object is represented by essentially two pieces of information: a list of the other \(p\)-adic objects on which it depends; and an \emph{approximation function} which takes as input a list of approximations of its dependencies at some epoch \(n\), and returns an approximation which is taken to be the approximation of the object at the same epoch.

Since the approximation function is only given the approximations of its dependencies at a given epoch \(n\), all it must do is return an approximation to the best precision it can given its inputs. It is not aiming to return an approximation to any specific precision.

For example, our representation of \(\QQ_p\) has no dependencies, and its representation at epoch \(n\) is the fixed-precision field \verb|pAdicField(p,2^n)| whose elements are of the form \(\pi^v (y + \pi^r \ZZ_p)\) for \(r \leq 2^n\). Hence the precision increases exponentially with \(n\), the intention being that one will, in a small number of epochs, be able to increase the precision of a \(p\)-adic object to some desired absolute precision. Since there are only a small number of possible epochs --- the user is highly unlikely to go beyond \(n=20\) --- the dependency-tracking framework should only be invoked relatively infrequently.

As we shall see in \S\ref{xp-sec-update-ii}, the dependency-tracking itself is also quite straightforward.

\subsection{Abstract base types}

As with \ExactpAdics{}, this package uses the abstract types \verb|StrPadExact| and \verb|PadExactElt| to represent \(p\)-adic structures (such a fields and rings) and elements respectively. However, these are now also both subtypes of \verb|AnyPadExact|, which represent any \(p\)-adic object:

\begin{lstlisting}
type AnyPadExact
attributes AnyPadExact: id, dependencies,
  approximations, get_approximation

type StrPadExact: AnyPadExact

type PadExactElt: AnyPadExact
attributes PadExactElt: parent
\end{lstlisting}

The \verb|id| attribute, as before, is a unique integer used to identify the object. It is assigned from a global counter, and so each object can only depend on objects with smaller \verb|id|. This is used to simplify dependency tracking.

The \verb|dependencies| attribute is a list of other \(p\)-adic objects (i.e. of type \verb|AnyPadExact|) on which this one directly depends.

The \verb|approximations| attribute is a list of approximations of the object. The object at position \(n\) in the list is the approximation of the object at epoch \(n\). It is analogous to the \verb|approximation| attribute in the \ExactpAdics{} package, except that we now record all approximations.

The \verb|get_approximation| attribute is the \emph{approximation function}, and is analogous to the \verb|update| function from the \ExactpAdics{} package. It is a function with two inputs: an epoch (a positive integer) and the list of approximations of the \verb|dependencies| at the given epoch. It must return the approximation of the object at the given epoch.

The \verb|parent| attribute of an element (a \verb|PadExactElt|) is the structure (a \verb|StrPadExact|) containing the element. The approximation of an element at epoch \(n\) must be en element of the approximation of the parent at epoch \(n\).

Figure \ref{xp-fig-strpadexact-ii} illustrates the relationships between these types and their attributes.

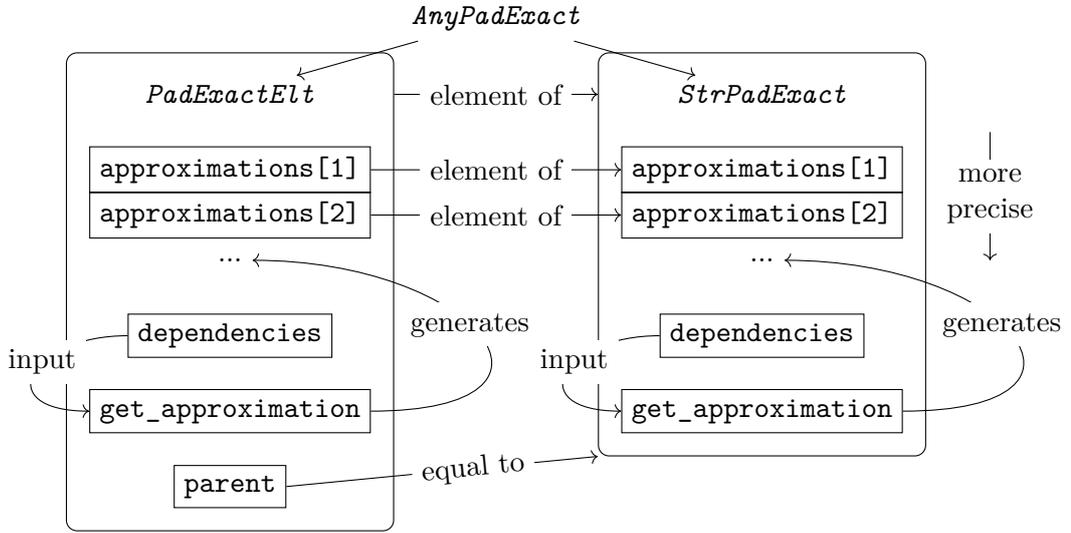
\begin{figure}
\centering
\begin{tikzpicture}[
  type/.style={font=\em\ttfamily},
  attr/.style={draw, font=\ttfamily},
  box/.style={draw, inner sep=0.3cm, rounded corners},
  label/.style={midway, fill=white},
]
\node(At) at (3.5,1) [type] {AnyPadExact};
\node(Et) at (0,-0) [type] {PadExactElt};
\node(Ea1) at (0,-1) [attr] {approximations[1]};
\node(Ea2) at (0,-1.6) [attr] {approximations[2]};
\node(Ea3) at (0,-2.2) [] {...};
\node(Ed) at (0,-3.2) [attr] {dependencies};
\node(Eg) at (0,-4.2) [attr] {get\_approximation};
\node(Ep) at (0,-5.2) [attr] {parent};
\node(St) at (7,-0) [type] {StrPadExact};
\node(Sa1) at (7,-1) [attr] {approximations[1]};
\node(Sa2) at (7,-1.6) [attr] {approximations[2]};
\node(Sa3) at (7,-2.2) [] {...};
\node(Sd) at (7,-3.2) [attr] {dependencies};
\node(Sg) at (7,-4.2) [attr] {get\_approximation};
\node(E)[box, fit={(Et) (Ea1) (Ea2) (Ea3) (Ed) (Eg) (Ep)}] {};
\node(S)[box, fit={(St) (Sa1) (Sa2) (Sa3) (Sd) (Sg)}] {};
\draw[->](At)--(St);
\draw[->](At)--(Et);
\draw[->](St -| E.east)--node[label]{element of}(Et -| S.west);
\draw[->](Ea1)--node[label]{element of}(Sa1);
\draw[->](Ea2)--node[label]{element of}(Sa2);
\draw[->](Ep.east)--node[label,sloped,pos=0.6]{equal to}(S.south west);
\draw[->](Ed.west) to[out=180,in=180,looseness=3] node[label,pos=0.4]{input}(Eg.west);
\draw[->](Eg.east) to[out=0,in=0,looseness=3] node[label,pos=0.55]{generates}(Ea3.east);
\draw[->](Sd.west) to[out=180,in=180,looseness=3] node[label,pos=0.4]{input}(Sg.west);
\draw[->](Sg.east) to[out=0,in=0,looseness=3] node[label,pos=0.55]{generates}(Sa3.east);
\draw[->](10,-0.5)--node[label,align=center]{more\\precise}(10,-2.2);
\end{tikzpicture}
\caption[Illustration of \texttt{AnyPadExact} and subtypes]{Illustration of the types \texttt{AnyPadExact}, \texttt{StrPadExact} and \texttt{PadExactElt}, their attributes and the relationships between them.}
\label{xp-fig-strpadexact-ii}
\end{figure}

\subsection{\texorpdfstring{\(p\)}{p}-adic fields}

The way in which \(p\)-adic fields and their elements are built on top of these base types is identical to the \ExactpAdics{} package:
\begin{lstlisting}
type FldPadExact[FldPadExactElt]: StrPadExact
attributes FldPadExact: xtype, prime, defining_polynomial
\end{lstlisting}
where the \verb|xtype| is either \verb|PRIME| indicating it is a prime \(p\)-adic field \(\QQ_p\), in which case the \verb|prime| attribute must be set to the prime \(p\), or else it is \verb|INERT| or \verb|EISEN| indicating an unramified or totally ramified extension, in which case the \verb|defining_polynomial| attribute must be set to the inert or Eisenstein defining polynomial.

The approximations of a \(p\)-adic field must be fixed-precision inexact \(p\)-adic fields, such as \verb|pAdicField(2,20)| in Magma (representing \(\QQ_2\), whose elements have relative precision at most 20). However, two such fields in Magma are considered to be different, even if they only differ in their precision, and yet we will need to coerce approximate \(p\)-adic numbers between different approximate \(p\)-adic fields representing the same exact field, which will be manual and slow. Hence we define
\begin{lstlisting}
attributes FldPadExact: infinite_precision_approximation
\end{lstlisting}
which is a semi-exact approximation of the field defined via a map, as described in \S\ref{xp-sec-terminology}. The \verb|approximations| are then fixed-precision versions of this infinite-precision field produced via the \verb|ChangePrecision| intrinsic in Magma. Since Magma now understands all of these approximations to come from a common underlying field, it performs coercion between them for free.

\subsection{Univariate polynomials}

The way in which rings of univariate \(p\)-adic polynomials and their elements are defined is again identical to the \ExactpAdics{} package:

\begin{lstlisting}
type RngUPol_FldPadExact[RngUPolElt_FldPadExact]
attributes RngUPol_FldPadExact: base_ring
\end{lstlisting}

It is defined by its \verb|base_ring|, a \(p\)-adic field (a \verb|FldPadExact|). The approximation of such a ring at epoch \(n\) is the univariate polynomial ring over the approximation at epoch \(n\) of the base ring.

\subsection{Examples}

\begin{example}
\label{xp-ex-add-ii}
Here we present an implementation of binary addition of \(p\)-adic numbers (cf. Example \ref{xp-ex-add2}).
\begin{lstlisting}
intrinsic '+' (x :: FldPadExactElt, y :: FldPadExactElt)
    -> FldPadExactElt
  z := New(FldPadExactElt)
  z`parent := x`parent
  z`dependencies := [x, y]
  z`get_approximation := function (n, xds)
    return xds[1] + xds[2]
  return z
\end{lstlisting}
\qedabove{}
\end{example}

In fact, most purely arithmetic functions are this simple to implement.

\begin{example}
Here we present an implementation of polynomial resultant. Since the resultant depends on the degree of the polynomials, we need to ensure that the approximations of the inputs have the correct degree using \verb|EnsureAllApproximationsAreFullDegree|, similar to as in Remark \ref{xp-rmk-ensurenonzero}.
\begin{lstlisting}
intrinsic Resultant (
  f :: RngUPolElt_FldPadExact,
  g :: RngUPolElt_FldPadExact)
    -> RngUPolElt_FldPadExact
  EnsureAllApproximationsAreFullDegree(f)
  EnsureAllApproximationsAreFullDegree(g)
  h := New(RngUPolElt_FldPadExact)
  h`parent := f`parent
  h`dependencies := [f, g]
  h`get_approximation := function (n, xds)
    return Resultant(xds[1], xds[2])
  return h
\end{lstlisting}
\qedabove{}
\end{example}

\subsection{Generating approximations}
\label{xp-sec-update-ii}

We now describe how we generate the \verb|approximations| of a \(p\)-adic object from its \verb|dependencies| and \verb|get_approximation| function.

Suppose we are given a \(p\)-adic object and an epoch \(n\), and we wish to compute the approximation of the object at the given epoch. The intrinsic \verb|BringToEpoch| does this for us:
\begin{lstlisting}
intrinsic BringToEpoch(x :: AnyPadExact, n :: RngIntElt)
  if #x`approximations lt n then
    for d in x`dependencies do
      BringToEpoch(d, n)
    xds := [d`approximations[n] : d in x`dependencies]
    xx := x`get_approximation(n, xds)
    x`approximations[n] := xx
\end{lstlisting}

First it checks if there is already an approximation at this epoch. If not, we run through the dependencies and bring these up to the same epoch recursively. Now we can construct a list of approximations of the dependencies at this epoch, pass this to \verb|get_approximation| to produce the required approximation, and update the \verb|approximations| list accordingly.

We also supply the intrinsic \verb"EpochApproximation" which returns the approximation at a given epoch:
\begin{lstlisting}
intrinsic EpochApproximation(x :: AnyPadExact, n :: RngIntElt)
  BringToEpoch(x, n)
  return x`approximations[n]
\end{lstlisting}

The true implementation of \verb|BringToEpoch| is slightly more complicated. The following subsections explain how.

\subsubsection{Saving the approximation}
Instead of simply saving the output of \verb|get_approximation| as a new approximation directly, we assign it using a generic intrinsic called \verb|SetApproximation|, which performs some checks. This includes checking that the approximation is of the right type; that, if it is an element, the approximation is an element of the approximation of its parent; and that the approximation is consistent with the current best approximation attached to the object.

Furthermore, our package assumes that if an object has an approximation at epoch \(n\), then it has approximations for all lower epochs. So what do we do if we are jumping from epoch 1 to 10, for example, how do we set the intermediate approximations? For each subtype of \verb|AnyPadExact|, there must be an intrinsic \verb|InterpolateEpochs| implemented which takes as input a \(p\)-adic object, a range of epochs, and an approximation at the top epoch. It must return a list of approximations for the intermediate epochs.

\begin{example}
The default implementation uses the approximation function to generate the intermediate values, provided we exceed the \verb"min_epoch". This is usually sufficient for structures. Note that the dependencies are guaranteed to be at the top epoch already.
\begin{lstlisting}
intrinsic InterpolateEpochs(x :: AnyPadExact,
    n1 :: RngIntElt, n2 :: RngIntElt, xx :: FldPadElt)
  if n1 ge x`min_epoch then
    return [x`get_approximation(n, xds)
      where xds := [d`approximations[n] : d in x`dependencies]
      : n in [n1..n2-1]]
  else
    error "not implemented: InterpolateEpochs with min_epoch>1"
\end{lstlisting}
\qedhigher
\end{example}

\begin{example}
For \(p\)-adic numbers, we coerce the approximation into the approximations of the parent field at the intermediate epochs.
\begin{lstlisting}
intrinsic InterpolateEpochs(x :: FldPadExactElt,
    n1 :: RngIntElt, n2 :: RngIntElt, xx :: FldPadElt)
  return [x`parent`approximations[n] ! xx : n in [n1..n2-1]]
\end{lstlisting}
\qedabove
\end{example}

Now if \verb|SetApproximation| is setting some approximation for a high epoch, it will use \verb|InterpolateEpochs| to fill in the gaps.

\subsubsection{Minimum epoch}

\begin{example}
Suppose we are implementing division of two \(p\)-adic numbers. Here is what looks like a reasonable implementation:
\begin{lstlisting}
intrinsic '/' (x :: FldPadExactElt, y :: FldPadExactElt)
  require IsDefinitelyNonzero(y)
  z := New(FldPadExact)
  z`parent := x`parent
  z`dependencies := [x, y]
  z`get_approximation := function (n, xds)
    return xds[1] / xds[2]
  return z
\end{lstlisting}
Note, however, that even though we checked that \verb|y| is non-zero, we are not guaranteed that all of its approximations are not weakly zero. If some of them are, then the division inside \verb|get_approximation| may raise an error.
\end{example}

To solve this, we have \verb|IsDefinitelyNonzero| return a second value, which is the smallest epoch at which an approximation for \verb|y| is not weakly zero. Note that since approximations may not become less precise as epoch increases, this implies that all approximations for \verb|y| are not weakly zero above this epoch. We can then set the new attribute
\begin{lstlisting}
attributes AnyPadExact: min_epoch
\end{lstlisting}
to this epoch.

The meaning of \verb|min_epoch| is that it is the smallest epoch for which the \verb|get_approximation| function should be called, and hence in our example, the division will only use non weakly zero approximations to \verb|y|.

To use \verb|min_epoch|, we simply need to insert the following line into \verb|BringToEpoch|
\begin{lstlisting}
n := Max(n, x`min_epoch)
\end{lstlisting}
which ensures that the epoch we are updating to is at least \verb|min_epoch|.

\begin{remark}
\label{xp-rmk-ensurenonzero}
For the particular case of division, and similar functions, we can take a different approach and implement it like so:
\begin{lstlisting}
intrinsic '/' (x :: FldPadExactElt, y :: FldPadExactElt)
  EnsureAllApproximationsAreNonzero(y)
  z := New(FldPadExactElt)
  z`parent = x`parent
  z`dependencies := [x, y]
  z`get_approximation := function (n, xds)
    return xds[1] / xds[2]
  return z
\end{lstlisting}
where, as the name suggests, the intrinsic \verb|EnsureAllApproximationsAreNonzero| ensures that all approximations of \verb|y| are not weakly zero. This is achieved by first calling \verb|IsDefinitelyNonzero| to check \verb|y| is nonzero and find an epoch at which its approximation is not weakly zero, and then by using \verb|InterpolateEpochs| and \verb|SetEpochs| to interpolate this approximation down to epoch 1.
\end{remark}

\subsubsection{Maximum epoch}
\label{xp-sec-max-epoch}

Analogous to \verb|min_epoch|, there is also
\begin{lstlisting}
attributes AnyPadExact: max_epoch
\end{lstlisting}
which is the maximum epoch at which \verb|get_approximation| should be called.

The intention here is that the user can set the maximum epoch on a \(p\)-adic object as a way of limiting the precision to which computations involving that object are performed.

The \verb|BringToEpoch| intrinsic is modified to insert a check that the target epoch is not greater than the \verb|max_epoch|, if it is set. If so, it will raise a precision error:
\begin{lstlisting}
if assigned x`max_epoch and n gt x`max_epoch then
  error "precision error: max_epoch exceeded"
\end{lstlisting}

We also supply the intrinsic \verb"CanBringToEpoch" which is the same as \verb"BringToEpoch" except that instead of raising a precision error when the \verb"max_epoch" is reached it returns false, and otherwise returns true to signal success.

\subsection{Precision strategies}
\label{xp-sec-strat-ii}

At present, we do not provide functionality analogous to the precision strategies (\S\ref{xp-sec-strat}) of the \ExactpAdics{} package. The only method for controlling precision currently available to the user is the \verb|max_epoch| attribute described in \S\ref{xp-sec-max-epoch}, which will cause an error to be raised if a computation requires too much precision.

Therefore currently, any functions which need to increase the precision of its inputs do so simply by trying each epoch in order. Hence, there are no \verb|Strategy| parameters in this package, and instead we can think of the sequence \(1,2,\ldots\) as the default strategy where the values are now epochs, not precisions.

\begin{example}
\label{xp-ex-valuation-ii}
\verb|Valuation| is implemented like this (cf. Example \ref{xp-ex-valuation}):
\begin{lstlisting}
intrinsic Valuation(x :: FldPadExactElt) -> Val_FldPadElt
  for n in 1,2,... do
    BringToEpoch(x, n)
    if not IsWeaklyZero(x) then
      return WeakValuation(x)
\end{lstlisting}
\qedabove
\end{example}

\section{Comparison of \ExactpAdics{} and \ExactpAdicsII{}}
\label{xp-sec-compare-ii}

\subsection{Complexity of updates}

Compare the procedures \verb|satisfy_dependencies| (\S\ref{xp-sec-getter-eval}) of \ExactpAdics{} and \verb|BringToEpoch| (\S\ref{xp-sec-update-ii}) of \ExactpAdicsII{}, which are the underlying means in each package of generating an approximation to a \(p\)-adic object.

In the latter, we satisfy each dependency recursively immediately. In the former, we perform a backwards pass to gather all dependencies together, followed by a forwards pass to satisfy dependencies.

The rationale for the behaviour of the former was discussed in \S\ref{xp-sec-getter-motivation}, and it comes down to the fact that the same \(p\)-adic object may appear multiple times in a dependency with different absolute precisions. By performing the backwards pass first, we can merge all such dependencies into one. On the other hand, in \verb"BringToEpoch" all dependencies are being brought to the same epoch, and therefore we can satisfy each dependency immediately without risk of it needing to be brought to a higher epoch later.

Additionally, in \ExactpAdics{}, satisfying a dependency is allowed to fail (i.e. the \verb|get_value| procedure of a \verb|Getter| is allowed to not return a value), which triggers a new backwards pass to find dependencies of this failed update, and an extra forwards pass will have to occur to satisfy these. Hence there is in principle no bound on the amount of dependency tracking required to update a single element, whereas in \ExactpAdics{} we do a single pass.

This is a necessary feature of the design of \ExactpAdics{}: because the \verb|update| function must update its target object to a given absolute precision, we must allow it the freedom to take a guess at the precision required of its dependencies, and then try a better guess if it turns out this was too low. This is because there are some operations where it is difficult or impossible to determine the dependency precisions in advance.

On the other hand, in \ExactpAdicsII{}, because there is a looser relationship between precisions and epochs, the \verb|get_approximation| function is not aiming for any specific precision. Instead, it simply needs to produce an approximation to the best precision it can.

\subsection{Number of updates}

In \ExactpAdicsII{}, an ``update'' can occur at each epoch. Since precisions are exponential in the epoch, then typically there are only a small number of epochs ever considered, rarely going beyond epoch 20. This limits the number of times the dependency tracking framework ever needs to consider a single object, and so the time spent doing dependency tracking is essentially a small constant times the number of variables in a computation.

On the other hand, in \ExactpAdics{} one can in principle increase the precision of an element by 1 many times, and each time the dependency tracking code will be invoked, so there is essentially no bound on the time spent doing this. To mitigate this, one could modify the package so that elements can only increase their precisions by large jumps, such as doubling each time.

\subsection{Implementing new functions}

To implement a new low-level operation in \ExactpAdics{}, such as addition of two \(p\)-adic numbers, requires implementing a \verb|Getter| which (a) can compute the precisions to which its dependencies are required; and (b) compute an approximation, given approximations of its dependencies. To do the same in \ExactpAdicsII{} only requires (b), and therefore implementing new functionality in the latter is often much quicker.

Furthermore, actually computing the dependency precisions can be slow:

\begin{example}
Let \(h(x) = f(x) g(x)\) be a product of two polynomials. Suppose we want to compute an approximation to \(h\) with the absolute precision of the \(k\)th coefficient (\(h_k\)) at least \(a_k\). Then we need \(f_i\) to absolute precision \(\max_j a_{i+j} - \val(g_j)\) and \(g_j\) to absolute precision \(\max_i a_{i+j} - \val(f_i)\).

Computing these absolute precisions is of the same order of complexity as performing the multiplication itself. On the other hand, the multiplication is implemented in a low level compiled language such as C, whereas our \ExactpAdics{} package is implemented in the high-level interpreted language Magma, and so computing these absolute precisions can be far more expensive.
\end{example}

\subsection{Precision optimality}

By design, all computations in \ExactpAdics{} are performed to as little precision as is possible to get the answer. In \ExactpAdicsII{}, we perform all computations starting from the same initial precision and keep doubling this precision as necessary. Hence the latter is not optimal in terms of precision used, but is typically within a factor of 2 of optimal.

It is possible for \ExactpAdicsII{} to be worse than this. Suppose \(x\) is cheap to compute approximations for, but loses a lot of precision along the way, so if it has an approximation in \verb|pAdicField(p,2^n)| then its precision is significantly less than \(2^n\). Also suppose that \(y\) is expensive to compute, and does not lose any precision. Let \(z = x+y\). Now because \(x\) loses a lot of precision, so does \(z\), and therefore computing an approximation to \(z\) requires a relatively high epoch. Computing the approximation to \(y\) at this epoch is expensive, but also unnecessary because it achieves this required precision at an earlier epoch.

In a sense, the epoch is not really a proxy for the precision of an element, but a proxy for the worst precision of all the dependencies of the element.

For most common applications, the amount of precision lost tends to be bounded and small as epoch increases, and so this effect is minimal.

\subsection{Precomputing dependencies}
\label{xp-sec-compare-optimize}

In \S\ref{xp-sec-optimize} we describe a generic optimization technique which can make updating the approximations of a selected \(p\)-adic object much quicker. This is done by pre-computing some of its dependency graph so that it can be traversed more efficiently.

This optimization opportunity is only possible in \ExactpAdicsII{}. Even if \ExactpAdics{} were redesigned to make the list of dependencies of an object explicit, so that a piece of the dependency graph could be precomputed, we would still need to do a backwards pass to find the minimal precision required of each dependency.

\subsection{Precision strategies}

In \ExactpAdics{}, whenever a function needs to increase the precision of an object in a non-canonical way, it does so according to a precision strategy (\S\ref{xp-sec-strat}), giving fine control over each precision tried.

On the other hand, \ExactpAdicsII{} currently has no such functionality other than setting the \verb|max_epoch| parameter on an object (\S\ref{xp-sec-strat-ii}). When a function needs to increase the precision of an object in a non-canonical way, it repeatedly increases the epoch by 1. In practice, precision strategies in \ExactpAdics{} will usually just repeatedly double the precision, which behaviour is almost the same as increasing the epoch by 1 in \ExactpAdicsII{}. In principle, the package could have strategies to control which epochs are used, but this is not yet implemented.

\subsection{Timings}
\label{xp-sec-timings}

\subsubsection{Dependency tracking}
\label{xp-sec-timings-addition}

In this section, we describe an experiment designed to stretch the dependency tracking capabilities of our packages. This involves performing a computation which involves thousands of intermediate variables, but the steps themselves are cheap to compute.

In this experiment, we define \(x_1 = 1, x_2 = 2 \in \QQ_2\) and for \(i=3,\ldots,10000\) we define \(x_i = x_{j_i} + x_{k_i}\) for some randomly chosen \(j_i,k_i \in \{1,2,\ldots,i-1\}\). Finally we define \(y = \sum_{i=1}^{10000} x_i\). We time how long it takes to compute \(y\) to absolute precision \(2^n\) for \(n=1,\ldots,16\), taking the total time --- this emulates a typical sequence of increasing the absolute precision of \(y\) 16 times. We repeat this 10 times with different random choices and take the mean.

This experiment is repeated using a number of different \(p\)-adic implementations, with the mean timings given in Table~\ref{xp-tbl-timings-addition}. Note that the random choices are made in advance and so are not timed, and we use the same random seed in each experiment, so precisely the same sequence of operations is being compared.

\begin{table}
\centering
\begin{tabular}{l rcrcr}
\hline
Experiment & \hspace{-5em}Time (sec) & & & & \\
\hline
(1) Builtin                         &   0.949 & & & & \\
(2) \ExactpAdics{}                  & 174.284 & = & 6.907 & + & 167.377 \\
(3) \ExactpAdicsII{}                &   7.330 & = & 0.400 & + &   6.930 \\
(4) \ExactpAdicsII{} (opt: default) &   7.047 & = & 0.528 & + &   6.519 \\
(5) \ExactpAdicsII{} (opt: fast)    &   2.006 & = & 0.460 & + &   1.546 \\
\hline
\end{tabular}
\caption[Timings for a highly dependent computation]{Timings for a highly dependent computation over different implementations, including two optimizations.}
\label{xp-tbl-timings-addition}
\end{table}

Experiment (1) uses the builtin inexact \(p\)-adics available in Magma, and so is a reasonable lower bound on what we can expect to achieve. Experiments (2) and (3) use the \ExactpAdics{} and \ExactpAdicsII{} packages, respectively. These timings are broken into two parts, the first part being the time to construct \(y\), and the second part being the time to increase its precision to \(2,4,\ldots,2^{16}\). We can see that the latter package outperforms the former significantly on both counts.

Experiments (4) and (5) are the same as (3), except we use the optimization techniques described in \S\ref{xp-sec-optimize} to make \(y\) directly depend only on \(x_1\) and \(x_2\). Experiment (4) uses the default version, which gives a small speed-up. Experiment (5) uses the ``fast'' version, which forgets the intermediate variables and uses the \verb|get_approximation| functions directly, and achieves a significant speed-up.

\subsubsection{Real-world example}

We compute the 2-part of the conductor of the hyperelliptic curve
\[C : y^2 = -2x^6-15x^4-37x^2-30\]
using our implementation of \cite{conductor} mentioned in \S\ref{xp-sec-intro}. This implementation can use either of our packages for its underlying \(p\)-adic computations. The curve has discriminant \(\Delta = - 2^{16} \cdot 3 \cdot 5\) and conductor \(N = 2^{10} \cdot 3 \cdot 5\).

When using \ExactpAdics{}, this takes 146 seconds, compared to 33 seconds for \ExactpAdicsII{}. The time spent in the dependency-tracking portion of code, which includes actually computing approximations, is 124 and 25 seconds respectively, with 22 and 12 seconds respectively left over to other computations.

With \ExactpAdics{}, this 124 seconds spent in dependency tracking is divided equally between generating approximations and tracking dependencies (this includes calling the update function and computing dependencies). About half of the latter is spent computing dependencies, most of the rest being logic comparing absolute precisions.

In fact, of the whole 146 seconds, 39 seconds is spent just constructing our representation of a valuation of a univariate polynomial. Individually this is fast, but we construct 240,000 of them throughout the algorithm. This demonstrates the benefit of using epochs instead of fine absolute precisions in \ExactpAdicsII{}.

Note that the number of times the dependency tracking framework is invoked is about 54,000 and 40,000 for the two packages. Given the same algorithms are used in both packages, we expect these numbers to be similar. In this case we do not appear to have the potential issue that the framework is invoked too often. The number of \(p\)-adic objects created is about 11,000 and 7,000 for the two packages.

\subsection{Conclusions}

Given the above arguments and evidence, we currently recommend the typical user to choose \ExactpAdicsII{} over \ExactpAdics{}.

On the other hand, if more of the internal workings were implemented at a lower level than the Magma language and optimized, then it may be that \ExactpAdics{} could be made comparably fast. Indeed, much of the comparative slowness in \ExactpAdics{} comes from the need for a lot of simple arithmetic to compute absolute precisions, which is typically slow in an interpreted language such as Magma.

\section{Additional structures}
\label{xp-sec-others}

So far we have described the representation of \(p\)-adic numbers and univariate polynomials over \(p\)-adic fields. We now briefly describe two more structures provided by the package.

\subsection{Multivariate polynomials}

A multivariate polynomial ring over a \(p\)-adic field is represented by the type \verb|RngMPol_FldPadExact| (analogous to the inexact type \verb|RngMPolElt| in Magma) which derives from \verb|StrPadExact|:

\begin{lstlisting}
type RngMPol_FldPadExact[RngMPolElt_FldPadExact]
attributes RngMPol_FldPadExact: base_ring, rank
\end{lstlisting}

Such a ring is defined by its \verb|base_ring|, an exact \(p\)-adic field (i.e. of type \verb|FldPadExact|), and by its \verb|rank|, the number of indeterminates.

An approximation of such a ring must be the multivariate \verb|PolynomialRing| of an approximation of the \verb|base_ring| of the same rank.

\subsection{Cartesian products}

The cartesian product of a number of exact \(p\)-adic structures is itself an exact \(p\)-adic structure, and has the type \verb|SetCart_PadExact| analogous to the type \verb|SetCart| for general cartesian products.

\begin{lstlisting}
type SetCart_PadExact[Tup_PadExact]
attributes SetCart_PadExact: components
\end{lstlisting}

Such a cartesian product is defined by its \verb|components|, a list of exact \(p\)-adic structures.

An approximation of this structure must be the cartesian product (a \verb|SetCart|) of an approximations of its components.

Why do we define this specialised form of cartesian products, when a general one exists already? The difference is that a standard tuple of exact \(p\)-adic values treats the component values as completely independent objects, whereas the exact tuple links them together in the sense that they have a single common update/approximation function. Therefore, the exact tuple is an appropriate choice for a collection of \(p\)-adic values which belong to some conceptually higher structure.

\begin{example}
Suppose we wish to implement a Hensel-lifting routine which takes as input a sequence \(F \in K[x_1,\ldots,x_n]^n\) of \(n\) multivariate polynomials of rank \(n\) over some \(p\)-adic field \(K\) and a sequence \(X \in K^n\) of \(n\) elements of \(K\) such that we can apply Hensel's lemma to deduce there is a root \(Y \in K^n\) of \(F\) close to \(X\), and returns the sequence \(Y\) (as in \S\ref{xp-sec-hensel-multiroot}).

To update the components of \(Y\) we perform a Hensel-lifting routine which is essentially some \(n \times n\) linear algebra depending on \(F\) and \(X\), the important point being that this computes all components of \(Y\) to some precision simultaneously; it is not possible to compute one component of \(Y\) to high precision in isolation. Therefore it makes sense to represent \(Y\) as a tuple with a single update function.

By comparison, if we represented \(Y\) as a sequence of independent values, then each component would still need to maintain its own approximation to the whole vector \(Y\) in order to perform Hensel lifting. Worse still, increasing the precision on one component would perform Hensel lifting, but then only update that one component even though the information is available to update all components. Therefore, increasing the precision of all components of \(Y\) would be \(n\) times too slow.
\end{example}

\section{Valuations}
\label{xp-sec-val}


In our packages, the valuation of a \(p\)-adic element \verb|PadExactElt| is intended to be the finest measure available of the valuation of the components of the element. Because there are many different types of \(p\)-adic elements (e.g. numbers, polynomials, tuples), there are as many different types of valuations, all needing to be represented somehow. There are some operations common to all valuations, such as addition, so we define a new abstract type to represent all types of valuation:
\begin{lstlisting}
type Val_PadExactElt
\end{lstlisting}
and we shall later define sub-types corresponding to each \(p\)-adic structure.

Note that the difference of two valuations is also a valuation, corresponding to the division of two \(p\)-adic elements with those valuations. Therefore, all kinds of precisions --- absolute, relative and baseline --- are also valuations.

In the packages, we use valuations of subtype of \verb|Val_PadExactElt| to represent all valuations (including weak valuations) and all precisions. In particular, the input to an update function is a valuation in this form, representing the intended absolute precision.

\begin{lstlisting}
attributes Val_PadExactElt: value
\end{lstlisting}

The \verb|value| field of a valuation contains the actual value of the valuation, whose representation is element-dependent.

\subsection{Valuations of \texorpdfstring{\(p\)}{p}-adic numbers}

Ordinarily we think of the valuation of a \(p\)-adic number as an integer, except that:
\begin{itemize}
\item The valuation of zero is not an integer, it takes the special value \(\infty\).
\item Multiplication of two valuations is not a useful concept: it has no description in terms of the arithmetic of \(p\)-adic numbers.
\item On the other hand addition and infimum do make sense: the sum of two valuations corresponds to the multiplication of \(p\)-adic numbers, and the minimum of two valuations corresponds to the addition of \(p\)-adic numbers.
\item Multiplication and division of a valuation by an integer or rational number also does make sense, since it corresponds to exponentiation of a \(p\)-adic number.
\item It is useful to be able to talk about the valuation of elements in an extension, and these may be rational numbers.
\item Subtraction is useful to define, since a relative or baseline precision is the difference of two valuations. In particular, we also need to include the symbol \(-\infty := 0 - \infty\).
\item Supremum is also a useful operation: if we increase the absolute precision of a \(p\)-adic number several times, then its final absolute value is the maximum of the intermediate absolute precisions.
\end{itemize}
We deduce that the standard ring of integers \((\ZZ,+,\times)\) is not a useful structure for valuations to reside in; instead, we define the set \(Z := \QQ \cup \{\pm\infty\}\), elements of which we represent with the type:

\begin{lstlisting}
type Val_FldPadElt: Val_PadExactElt
\end{lstlisting}

The \verb|value| attribute is either an integer (a \verb|RngIntElt| in Magma), a rational number (a \verb|FldRatElt|) or \(\pm\infty\) (a \verb|Infty|).

The following operations are supported:
\begin{itemize}
\item Addition: Defined for all pairs of elements of \(Z\), except \(\infty + (-\infty)\) is left undefined and will cause an error.
\item Subtraction: Defined for all pairs of elements of \(Z\). In particular, \(\infty - \infty\) is defined to be 0; this is because the \(p\)-adic number 0 represented to infinite \(p\)-adic absolute precision has infinite weak valuation, and so \(\infty-\infty\) should be its relative precision, which is 0. While an arbitrary collection of additions and subtractions is not associative by these definitions, in practice if subtraction is only used to compute precisions, then the results will be well-defined.
\item Infemum and supremum (which are the operations \verb|meet| and \verb|join| in the Magma language).
\item Multiplication and division by rational numbers (which we term \emph{scaling}).
\item Equality, inequality, and orderings \(=, \neq, \leq, <, \geq, >\). In particular, \(Z\) is totally ordered.
\item An operation called \verb|diff| which is defined as follows: \verb|x diff y| is \verb|x| if \verb|x > y| and otherwise is \(-\infty\). Note that it is the lowest valuation \verb|z| such that \verb|z join y = x join y|. It has a natural interpretation in our context: if we require an element to have precision \verb|x| and its current precision is \verb|y| then \verb|x diff y| is the lowest valuation \verb|z| such that increasing the precision to \verb|z| suffices. Whilst defining such an operation for single \(p\)-adic numbers may seem like overkill, it turns out to be useful for aggregates.
\end{itemize}

\subsection{Valuations of aggregate structures}

All other \(p\)-adic structures in the package are aggregate structures, in the sense that they represent, perhaps recursively, a collection of \(p\)-adic numbers. As defined at the top of the section, a valuation in the package is the finest possible description of the components of a \(p\)-adic element, and therefore we represent valuations of an aggregate as an analogous aggregate of valuations. Specifically:

\begin{itemize}
\item Univariate polynomials: A polynomial \(f(x) = \sum_{i=0}^\infty f_i x^i \in K[x]\) over a \(p\)-adic field \(K\) may be more simply thought of as the infinite sequence \((f_0,f_1,\ldots)\) of its coefficients, which is zero for all but finitely many places. Correspondingly, its valuation we represent as the infinite sequence \((\val(f_0),\val(f_1),\ldots)\), which is \(\infty\) at all but finitely many places. If we subtract two such valuations pointwise, the result is an infinite sequence which is 0 at all but finitely many places. Most generally then, a valuation of a univariate polynomial is an infinite sequence which takes the same value at all but finitely many places.

In the package, we define the new type \verb|AssocDflt| which represents an associative array with a default value; that is, it has a default value so that if a key is not in the array, then the value of the array at that key is the default. These are useful for representing functions which are constant at all but finitely many places.

Valuations of univariate polynomials are represented by the type:
\begin{lstlisting}
type Val_RngUPolElt_FldPad: Val_PadExactElt
\end{lstlisting}
whose \verb|value| is a default associative array \verb|AssocDflt| whose keys are non-negative integers \(i\) and whose values are \verb|Val_FldPadElt|s.

\item Multivariate polynomials: A polynomial \[f(x_1,\ldots,x_r) = \sum_{e \in \{0,1,\ldots\}^r} f_e x_1^{e_1} \cdots x_r^{e_r} \in K[x_1,\ldots,x_r]\] of rank \(r\) over \(K\) can be thought of as the map \(e \mapsto f_e\) taking exponent vectors to the corresponding coefficient. As with univariate polynomials, this map is zero almost everywhere. In analogue with univariate polynomials, we represent the valuation of a multivariate polynomial with the type:
\begin{lstlisting}
type Val_RngMPolElt_FldPad: Val_PadExactElt
\end{lstlisting}
whose \verb|value| is a default associative array \verb|AssocDflt| whose keys are exponent vectors \(e\) and whose values are the corresponding \verb|Val_FldPadElt|s.

\item Tuples: Valuations of tuples \verb|Tup_PadExactElt| of exact \(p\)-adic elements are represented by the type:
\begin{lstlisting}
type Val_Tup_PadExactElt: Val_PadExactElt
\end{lstlisting}
whose value is a corresponding tuple of valuations, representing the valuations of the components of the tuple.
\end{itemize}

These valuations all support the following operations:
\begin{itemize}
\item Addition, subtraction, scaling, infimum (\verb|meet|), supremum (\verb|join|), \verb|diff|: These are all defined point-wise.
\item Equality and inequality: two valuations are equal iff they are equal point-wise.
\item Ordering: two valuations are ordered if that ordering applies point-wise.
\end{itemize}

Note that while the set \(Z = \QQ \cup \{\pm\infty\}\) of valuations for \(p\)-adic numbers is totally ordered --- and this ordering is respected by the ordering, infimum and supremum operations --- the valuations for aggregate \(p\)-adic elements are only partially ordered. For example two tuples in \(\QQ_2^2\) may have valuations \((1,2)\) and \((2,1)\) and so are not ordered relative to each other, or they may have valuations \((1,2) < (2,2)\). This partial ordering is respected by infimum and supremum; for example \verb|x join y| is the unique smallest valuation greater than or equal to both \verb|x| and \verb|y|.

\begin{example}
Suppose a univariate polynomial of degree 5 is known to absolute precision \(x = (3,5,8,10,13,2,\infty,\infty,\ldots)\). The infinite precisions indicate that we know that coefficients 6 upwards are precisely zero. Also suppose we want to increase its absolute precision to at least \(y = (10,10,\ldots)\). Then it suffices to increase it to \begin{align*}y \texttt{ diff } x &= (10 \texttt{ diff } 3, 10 \texttt{ diff } 5, 10 \texttt{ diff } 8, 10 \texttt{ diff } 10, \\ &\qquad 10 \texttt{ diff } 13, 10 \texttt{ diff } 2, 10 \texttt{ diff } \infty, \ldots) \\ &= (10, 10, 10, -\infty, -\infty, 10, -\infty, \ldots)\end{align*} and so we see it suffices to only increase the precisions of coefficients 0, 1, 2 and 5.
\end{example}

\section{Additional features}
\label{xp-sec-features}

We now describe some of the high-level features available in our packages, including notes on how they are implemented. The majority of these features are in both packages, but any pseudo-code in this section will be as in \ExactpAdicsII{}.

\subsection{Precomputing dependencies}
\label{xp-sec-optimize}

Suppose \(d=(d_1,\ldots,d_k)\) are \(p\)-adic objects, and \(x\) is some complicated expression in \(d\), such as in \S\ref{xp-sec-timings-addition}. Hence \(x\) does not depend directly on \(d\), it depends on some intermediate expressions which recursively ultimately depend on just \(d\). To compute an approximation for \(x\) requires traversing its dependency graph, including all these intermediate expressions, which will be time-consuming. If we do not care about the intermediate expressions, then it could be more efficient to compute approximations to \(x\) directly from \(d\).

In the \ExactpAdicsII{} package, we provide an intrinsic \verb|WithDependencies| which takes a \(p\)-adic object \(x\) and a list \(d\) of other \(p\)-adic objects and returns a copy of \(x\) whose direct \verb|dependencies| are precisely \(d\). The basic idea is that we pre-compute the piece of the dependency graph between \(x\) and \(d\), which the \verb|get_approximation| function can traverse efficiently.

Specifically, starting from \(x\), we recursively traverse its dependencies, gathering them together to form the set of all of its dependencies. Whenever we reach a dependency lying in \(d\), we terminate that branch of the recursion, so that we only find the dependencies between \(x\) and \(d\). Next, we sort these dependencies by \verb|id| into a list. Since the \verb|id|s are assigned sequentially, this also sorts according to dependency.

With its default behaviour, \verb|WithDependencies| also incorporates information about the \verb|min_epoch| of each dependency into this list: specifically the list is now a list of pairs \((y,m)\) where \(m\) is the maximum \verb|min_epoch| of \(y\) or anything depending on \(y\). Having precomputed this list, we can define \verb|get_approximation| to traverse this list in order: given an epoch \(n\), for each \((y,m)\) in the list, we compute an approximation to \(y\) at epoch \(\max(n,m)\) from its dependencies, which will already be at this epoch, and update \(y\) accordingly.

\verb|WithDependencies| also has a \verb|Fast| parameter which performs a more aggressive optimization. Note that the default behaviour still explicitly deals with all intermediate dependencies \((y,m)\), and in particular each such \(y\) is updated in the usual manner, which involves a number of consistency checks. The ``fast'' version ultimately forgets these dependencies entirely and instead just remembers the \verb|get_approximation| function attached to each one. These are called directly, one by one, with the approximations they return just appended to a temporary list, which is used as input when calling the next one, and so on. The last item in this list will be the approximation to \verb|x| returned by \verb|get_approximation|. As a result, there is no cacheing or consistency checking of each intermediate approximation, which can be a significant speed-up. The final answer, which is used to update \(x\), is still checked in the usual manner so we do not lose any safety.

Note that because the ``fast'' algorithm does not allow cacheing of intermediate variables, the \verb|min_epoch| of the created object must be the maximum of the \verb|min_epoch|s of all dependencies. Similarly its \verb|max_epoch| must be the minimum of those of its dependencies.

Furthermore, the ``fast'' algorithm assumes that the \verb"approximations" of the intermediate variables are all as produced by \verb"get_approximation". Therefore division, which can change the approximations of its dependencies (Remark~\ref{xp-rmk-ensurenonzero}), should not be an intermediate expression. For this reason, the \verb"Fast" parameter is false by default because it is not guaranteed to be safe. We also give division (and other intrinsics with the same issue) a \verb"Safe" parameter which, when true, does not use this trick and is therefore safe to be an intermediate expression, at the cost of a potentially higher \verb"min_epoch".

Timings demonstrating the benefits of using this optimization technique are given in \S\ref{xp-sec-timings-addition}. As noted in \S\ref{xp-sec-compare-optimize}, this generic optimization is not possible in \ExactpAdics{}.

\subsection{Valuation comparison}

A common \(p\)-adic operation is to compare the valuation of a \(p\)-adic number \(x\) with some given valuation \(v\). Consider the following code:
\begin{lstlisting}
if Valuation(x) gt 10 then
  ...
\end{lstlisting}
The first thing this does is compute the valuation of \(x\) precisely, and then compare the answer with 10. However, this is overkill: since there is no canonical way to increase the precision of \(x\) in order to find its valuation (which may be very high), then \verb|Valuation| will proceed according to some precision strategy, and therefore could never return an answer, or could raise a precision error.

We provide the following intrinsic:
\begin{lstlisting}
intrinsic ValuationGe(x, n)
  IncreaseAbsolutePrecision(x, n)
  return WeakValuation(x) ge n
\end{lstlisting}
so that \verb|ValuationGe(x, n)| is functionally very similar to \verb|Valuation(x) ge n| except that now there is a canonical way to increase the precision of \(x\) in order to get the answer, and it is guaranteed to produce a result with as little precision as required.

In reality, the definition of \verb|ValuationGe| is made a little more complex by checking if the answer is already known without increasing the precision of \(x\).

We similarly provide analogues \verb|ValuationEq|, \verb|ValuationNe|, \verb|ValuationLt|, \verb|ValuationGe| and \verb|ValuationGt| for the other comparison operators.

\subsection{Residue class fields and higher quotients}

Since Magma's inexact \(p\)-adics includes some functionality around residue class fields, we make similar functionality available in our package.

The intrinsic function \verb|ResidueClassField| takes as input an exact \(p\)-adic field \(K\) (type \verb|FldPadExact|) and returns its residue class field \(\FF\) (type \verb|FldFin|) and the quotient map \(q : \cO \to \FF\).

This is implemented by computing the residue class field of the \verb|approximation| field of \(K\) (i.e. \verb|ResidueClassField(K`approximation)|), which returns \(\FF\) and the quotient map \(\tilde{q} : \tilde \cO \to \FF\) where \(\tilde \cO\) is the integer ring of the \verb|approximation| field. Then \(q\) may be defined in terms of \(\tilde{q}\): given \(x \in K\), increase the absolute precision of \(x\) to at least 1, and then call \(\tilde{q}(\tilde{x})\).

The quotient map \(\tilde{q}\) also comes with a partial inverse, an embedding \(\tilde{q}^{-1} : \FF \into \tilde{\cO}\), which we similarly extend to a partial inverse \(q^{-1} : \FF \into \cO\). In this case, \(q^{-1}(x)\) is always given to absolute precision 1, and cannot have its absolute precision increased; in a sense, it refuses to choose among the many possible pre-images. In order to force such a choice, the intrinsic \verb|WeakApproximation| is provided, which takes as input an exact \(p\)-adic number, and returns another exact \(p\)-adic number which is equal to the input up to the precision of the input.

In a completely analogous manner, the intrinsic \verb|Quotient(K, n)| returns the ring \(\cO / \pi^n \cO\) and the quotient map \(q\), which again has a partial inverse. Hence \verb|Quotient(K, 1)| and \verb|ResidueClassField(K)| are equivalent, except that the latter represents the result as a field, and not a more general ring.

\subsection{Completions of number fields}
\label{xp-sec-completions}

Magma's inexact \(p\)-adics includes some functionality around taking completions of number fields at finite primes, so we make similar functionality available in our package.

The procedure \verb|ExactCompletion| takes as input a number field \(F\) and a finite place \(\fp\) of \(F\), and returns the completion \(K:=F_\fp\) as an exact \(\fp\)-adic field, and the embedding map \(e : F \into K\).

This is implemented around the builtin intrinsic \verb|Completion| which takes the same inputs, and returns the completion \(\tilde{K}\) as a semi-exact \(p\)-adic field, and the embedding \(\tilde{e} : F \into \tilde{K}\). Then \(K\) is simply an exact \(p\)-adic field whose approximation is \(\tilde{K}\), and \(e : F \into K\) returns an element whose update function uses \(\tilde{e}\) to embed the input element of \(F\) into \(\tilde{K}\) to sufficiently high precision.

\subsection{Newton polygons}

The following definitions and results are standard, if not the notation.

\begin{definition}
If \(f(x) = \sum_{i=0}^d f_i x^i \in K[x]\) is a polynomial over a \(p\)-adic field \(K\), then its \emph{Newton polygon} \(\cN(f)\) is the lower convex hull in \(\QQ \times \QQ\) of the points \((i, \val(f_i))\). It can also be interpreted as the graph of a function \([0,d] \to \QQ\), also denoted by \(\cN(f)\). By definition, this function is continuous, convex and piece-wise linear. If \(\cF\) is a face of the Newton polygon, i.e. a line segment from \((i_0,v_0)\) to \((i_1,v_1)\), then its \emph{width} is \(w(\cF)=w=i_1-i_0\) and its \emph{slope} is \(s(\cF) = \tfrac{v_1-v_0}{i_1-i_0}\). Writing \(s(\cF) = -\tfrac{h}{e}\) in lowest terms, then the \emph{ramification degree} of the face is \(e(\cF)=e\), and the \emph{residual polynomial} is \(r(\cF)(x) = \sum_{i=0}^{w/e} \overline{f_{ie+i_0} \pi^{ih-v_0}} \in \FF_K[x]\).
\end{definition}

\begin{lemma}
\label{xp-lem-newtonpgon}
If \(\cF\) is a face of \(\cN(f)\), then \(f\) has precisely \(w(\cF)\) roots in \(K^{\mathrm{alg}}\) of valuation \(-s(\cF)\). Writing \(-s(\cF)=h/e\) in lowest terms, if \(r\) is such a root, then \(r^e \pi^{-h}\) has valuation 0 and \(r(\cF)(\overline{r^e \pi^{-h}}) = 0\). Furthermore the roots \(r\) of valuation \(h/e\) are in \(e\)-to-1 correspondence with roots (possibly repeated) of \(r(\cF)(x)\) via \(r \mapsto \overline{r^e \pi^{-h}}\).
\end{lemma}

Hence the Newton polygon and related quantities provide much information about the roots of a polynomial, and so are invaluable in scenarios such as root-finding or factorization of polynomials.

We provide an intrinsic
\begin{lstlisting}
intrinsic NewtonPolygon(f :: RngUPolElt_FldPadExact
  : Support:=<0,Degree(f)>)
    -> NwtnPgon
\end{lstlisting}
which takes as input a \(p\)-adic polynomial \verb|f| and returns its Newton polygon. Since computing this involves computing the valuations of some of its coefficients, which may initially be weakly zero, it takes a \verb|Strategy| parameter. It also takes a \verb|Support| parameter which is a pair of integers representing a range, and the returned value will be a sub-polygon of the full Newton polygon supported on at least this range; this can be useful if, for example, the polygon might have a single root at 0, and so it suffices to get the piece of the Newton polygon on \([1,\infty)\).

The Newton polygon is computed as follows. We loop through precisions in the \verb|Strategy| and for each one, compute a corresponding approximation \verb|xf| of \verb|f|. We compute the \emph{lower weak Newton polygon} of \verb|xf|, defined to be the lower convex hull of the points \((i,w_i)\) where \(w_i\) is the weak valuation of the \(i\)th coefficient of \verb|xf|. We also compute the \emph{upper weak Newton polygon} of \verb|xf|, defined to be the lower convex hull of the points \((i,w_i)\) such that the \(i\)th coefficient of \verb|xf| is not weakly zero, and therefore \(w_i=\val(f_i)\). The lower weak Newton polygon lies below the Newton polygon, which in turn lies below the upper weak Newton polygon. Therefore if the weak polygons overlap anywhere, then that overlap is a section of the Newton polygon (see Figure \ref{xp-fig-newtonpgon}). If this section includes all of the \verb|Support| then we are done, otherwise we move on to the next precision in the strategy.

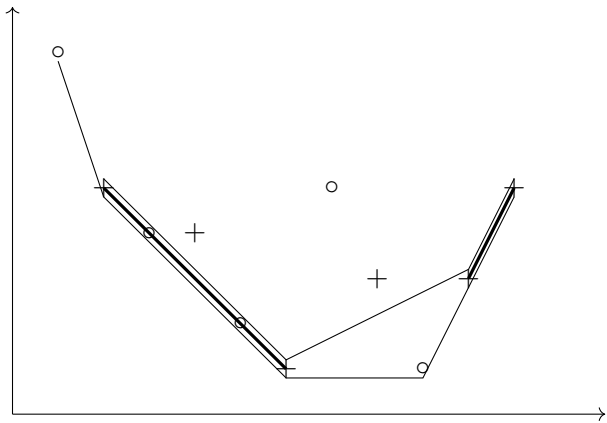
\begin{figure}
\centering
\begin{tikzpicture}[scale=0.6]
\draw[->](-1,-1)--(-1,8);
\draw[->](-1,-1)--(12,-1);
\foreach \x/\y in {0/7,2/3,4/1,6/4,8/0}
  \node(\x\y) at (\x,\y){\(\circ\)};
\foreach \x/\y in {1/4,3/3,5/0,7/2,9/2,10/4}
  \node(\x\y) at (\x,\y){\(+\)};
\draw([yshift=-0.2cm]07.center)--([yshift=-0.2cm]14.center)--([yshift=-0.2cm]50.center)--([yshift=-0.2cm]80.center)--([yshift=-0.2cm]104.center);
\draw[very thick](14.center)--(50.center);
\draw[very thick](92.center)--(104.center);
\draw([yshift=0.2cm]14.center)--([yshift=0.2cm]50.center)--([yshift=0.2cm]92.center)--([yshift=0.2cm]104.center);
\end{tikzpicture}
\caption[Computation of Newton polygon]{Computation of a section of a Newton polygon (heavy line) from lower and upper weak Newton polygons. Circles indicate the weak valuations of weakly zero coefficients, crosses indicate valuations of non weakly zero coefficients. Observe that since each end of the leftmost piece of the Newton polygon is at a vertex of the lower polygon, then these must also be vertices of the Newton polygon; contrast with the rightmost piece, in which the face could extend further to the left.}
\label{xp-fig-newtonpgon}
\end{figure}

\subsection{Ramification polygons and transition functions}
\label{xp-sec-ramification}

The ramification filtration of \(\Gal(L/K)\), the Hasse-Herbrand transition function and the upper-numbering of ramification groups are all standard, and appear for instance in Serre \cite[Ch. IV]{SerLF}. The theory extends to non-Galois extensions \cite{Helou}, which we summarise now.

\begin{definition}
Given a finite extension \(L/K\) of \(p\)-adic fields, its \emph{Galois set} \(\Gamma(L/K)\) is the set of \(K\)-embeddings of \(L\) into a normal closure --- this is a generalization of the Galois group. For \(\sigma \in \Gamma\), we define \(\val(\sigma) := \min_{x \in \cO_L} \val_L(\sigma x - x)\) and \(\Gamma_v := \{\sigma \in \Gamma \,:\, \val(\sigma) \geq v \}\) for \(v \geq 0\). The \emph{(lower) ramification breaks of \(L/K\)} are the \(v\) at which the function \(v \mapsto \abs{\Gamma_v}\) is discontinuous. We define the \emph{transition function} \[\phi_{L/K}(v) = \frac{1}{e(L/K)} \int_0^v \abs{\Gamma_t} dt\] which is continuous, piecewise linear, increasing and hence bijective \([0,\infty) \to [0,\infty)\), and letting \(\psi_{L/K}\) be its inverse, we define \(\Gamma^u = \Gamma_{\psi(u)}\). This defines the \emph{upper ramification numbering}. We define \(L^u = L_v\) to be the fixed field of \(\Gamma^u = \Gamma_v\) (where \(u=\phi(v)\)).
\end{definition}

The following lemma summarizes some key apsects of the theory. In particular, the upper numbering is well-behaved under changing the top field and fixing the base field, much in the way that the lower numbering is well-behaved under changing the base field. It also shows that the Galois correspondence generalizes to the sets \(\Gamma^u\).

\begin{lemma}[{\cite[Prop. 2, Rmk. 3, Prop. 3]{Helou}}]\hfill
\begin{enumerate}[noitemsep,nolistsep,label={\upshape(\alph*)}]
\item If \(M/L/K\) then \(\phi_{M/K} = \phi_{L/K} \circ \phi_{M/L}\).
\item Also \(\Gamma_{L/K}^u = \{\sigma|_L \,:\, \sigma \in \Gamma_{M/K}^u\}\), and in particular \(\Gamma_{L/K}^u\) are restrictions of elements of \(\Gal(L/K)^u\).  
\item \((L:L^u) = \abs{\Gamma^u}\) and so in particular \(L^u\) is the subfield of \(L\) fixed by \(\Gal(L/K)^u\).
\end{enumerate}
\end{lemma}

Computing quantities such as the transition function and upper/lower ramification breaks of an extension \(L/K\) is therefore of use when considering the Galois action of inertia or higher ramification groups. To compute these, we use ramification polygons, detailed decriptions of which appear in \cite[\S4--5]{GP} and \cite[\S3]{PS}. We summarize the key points here.

\begin{definition}
Suppose \(U/K\) is unramified, degree \(d\), \(f(x) \in U[x]\) is Eisenstein degree \(e\), defining the totally ramified \(L/U\), with uniformizer \(\pi \in L\) such that \(f(\pi)=0\). Then the \emph{ramification polygon of \(L/K\)} is the Newton polygon of the polynomial \(f(x + \pi)\) (which is supported on \([1,e]\)) with an additional horizontal face supported on \([e,ed]\).
\end{definition}

\begin{lemma}
The lower ramification breaks of \(L/K\) are \(v\) where \(-v\) is a slope of a face of the ramification polygon. The corresponding \(\abs{\Gamma_v}\) is the abscissa of the right hand vertex of the corresponding face. Letting \(v_0=0<\ldots<v_t\) be the lower breaks in sorted order and \(s_i = \abs{\Gamma_{v_i}}\), and letting \(u_0=0<\ldots<u_t\) be the upper breaks (i.e. \(u_i = \phi_{L/K}(v_i)\)) then \[\frac{u_{i+1} - u_i}{v_{i+1} - v_i} = \frac{s_i}{e(L/K)}\] gives a means to compute any one of these three sequences from the other two.
\end{lemma}

\begin{proof}
Since \(\cO_L = \cO_U[\pi]\), for \(\sigma \in \Gamma(L/U) = \Gamma(L/K)_1\) we have \(\val(\sigma) = \val(\sigma(\pi)-\pi) > 0\). Now \(\sigma(\pi)-\pi\) are precisely the roots of \(f(x-\pi)\), and so by Lemma \ref{xp-lem-newtonpgon} their valuations correspond to faces of the ramification polygon. Specifically, if \(-v\) is the slope of the face and \(w\) its width, then there are \(w\) elements \(\sigma \in \Gamma(L/U)\) such that \(\val(\sigma) = v\). Accumulating these widths from the left gives the sizes of \(\Gamma_v\), as claimed, for \(v>0\). The extra horizontal face by construction has slope 0 and vertex at \(ed=(L:K)=\abs{\Gamma}=\abs{\Gamma_0}\). The formula relating \(v_i, s_i, u_i\) follows from the definition of \(\phi_{L/K}\) as an integral.
\end{proof}

Hence the slopes and abscissa of vertices of faces of the Newton polygon correspond to \((v_i,s_i)\) and the vertices of the transition function correspond to \((v_i,u_i)\), and there is a bijective correspondence between these sequences. Therefore we can compute transition functions from Newton polygons and vice versa, provided we represent the transition function by its vertices. We introduce a new type to do so:
\begin{lstlisting}
type HasseHerbTransFunc
attributes HasseHerbTransFunc: vertices
\end{lstlisting}

It is easy to evaluate the transition function at a given \(v\) or its inverse at \(u\) by interpolating between the vertices. If we have the transition functions \(\phi_{L/K}\) and \(\phi_{M/L}\), then \(\phi_{M/K} = \phi_{L/K} \circ \phi_{M/L}\) has as its lower breaks the union of: (a) the lower breaks of \(\phi_{M/L}\); and (b) \(\phi_{L/K}^{-1}\) applied to the lower breaks of \(\phi_{L/K}\). The upper breaks are similar, and hence we have the vertices of \(\phi_{M/K}\) and therefore deduce a function to compose transition functions.

Now if we are given such an \(M/L/K\) say, with \(M/L\) and \(L/K\) each defined by an Eisenstein polynomial over an unramified extension, then we can compute the ramification polygons of \(M/L\) and \(L/K\) via the definition above. From this, we can compute the transition functions \(\phi_{M/L}\) and \(\phi_{L/K}\). From these and the composition routine described above, we can compute \(\phi_{M/K}\) and from this compute the ramification polygon of \(M/K\). In this manner, we deduce an intrinsic \verb|RamificationPolygon| to compute the ramification polygon of an arbitrary extension of \(p\)-adic fields and \verb|TransitionFunction| to compute the corresponding transition function.

\subsection{Hensel's lemma for univariate root-finding}
\label{xp-sec-hensel}

Recall Hensel's classic lemma.

\begin{lemma}[Hensel]
Suppose \(f(x) \in \cO[x]\), \(a \in \cO\) such that \(v(f(a)) \geq s > 0 = v(f'(a))\). Then there exists a unique \(b \in K\) such that \(f(b) = 0\) and \(v(a-b) \geq s\). More precisely, defining \(a' := a - f(a)/f'(a)\) then \(v(f(a')) \geq 2s\) and \(v(f'(a')) = 0\), so iterating \(a \mapsto a'\) then \(a \to b\).
\end{lemma}

We refer to the iteration process in Hensel's lemma as ``Hensel lifting''. It can be generalized to non-integral inputs:

\begin{lemma}
Suppose \(f(x) \in K[x]\), where \(K\) is a \(p\)-adic field, and \(a \in K\) such that among all roots \(b\) of \(f\), \(v(a-b)\) is maximised precisely once. Then iterating \(a \mapsto a - f(a)/f'(a)\) yields \(a \to b\).
\end{lemma}

\begin{proof}
The generalization is actually reducible to the original version.

Consider the polynomial \(f(x+a)\). Its roots are \(b-a\) where \(b\) is a root of \(f\), and so its Newton polygon measures the number of times each \(v(a-b)\) occurs. Hence the hypothesis is equivalent to saying that the first face of the Newton polygon of \(f(x+a)\) has width 1.

Suppose this is true, then in particular the first face has integral slope and so there exist \(j,k \in \ZZ\) so that \(g(x) := \pi^j f(\pi^k x + a)\) has integral coefficients, \(\val(g_0) > 0\) and \(\val(g_1) = 0\). Note that \(g_0 = g(0)\) and \(g_1 = g'(0)\) so the original version of Hensel's lemma applies to \(g\) and \(0\). By linearity, Hensel lifting on \(g\) is equivalent to Hensel lifting on \(f\).
\end{proof}

\begin{remark}
Krasner's lemma is a corollary of this form of Hensel's lemma.
\end{remark}

We provide an intrinsic \verb|IsHenselLiftable| which takes as input a polynomial \(f(x) \in K[x]\) and an element \(a \in K\) and returns true if this generalized version of Hensel's lemma can be applied to find a root \(b\) of \(f\) close to \(a\). If so, it also returns that root.

The algorithm proceeds by computing \(f(x+a)\) to sufficient precision to see if the first face of its Newton polygon has width 1 or not. If so, then the returned root has as its initial approximation the approximation of \(a\) truncated to a certain precision determined by Hensel's lemma, and its update function performs the Hensel lifting iteration above.

\begin{lstlisting}
intrinsic IsHenselLiftable(
  f :: RngUPolElt_FldPadExact,
  a :: FldPadExactElt)
    -> BoolElt, FldPadExactElt

  // first determine if Hensel's lemma is applicable
  // try successively precise approximations
  for n in 1,2,... do
    // get an approximation of f and a
    xf := EpochApproximation(f, n)
    xa := EpochApproximation(a, n)
    // approximate f(x+a)
    xf2 := Evaluate(xf, x + xa)
    // this Newton polygon is computed from the *weak*
    // valuations, so is not necessarily correct
    np := NewtonPolygon(xf2)
    face := Faces(np)[1]
    // if the first face has width 1 and the right hand
    // vertex is correct, then there really is a face of
    // width 1
    if Width(face) eq 1
    and not IsWeaklyZero(Coefficient(xf, 1))
    then break
    // if the face has higher width, and both vertices
    // are correct, then there really is a face of this
    // width
    elif Width(face) ne 1
    and not IsWeaklyZero(Coefficient(xf, 0))
    and not IsWeaklyZero(Coefficient(xf, EndVertices(face)[2][1]))
    then return false
    // else we cannot conclude whether the first face
    // has width 1 or not
    else continue
  // if we get this far, then a is Hensel liftable
  // we omit the implementation of Hensel lifting
  root := ...
  return true, root
\end{lstlisting}

\subsection{Univariate root finding I}
\label{xp-sec-roots}

Magma provides an intrinsic \verb|Roots| to find all of the roots of a univariate polynomial over an inexact \(p\)-adic field. As discussed in \S\ref{xp-sec-compare}, perhaps confusingly these are roots ``up to precision'', so for example given the polynomial \(x^2 + 2^{10} \ZZ_2\) over \(\QQ_2\), it will return the root \(0 + 2^{10} \ZZ_2\) with multiplicity 2. In a sense this is misleading, because it could be that the polynomial is acutally \(x^2 + 2^{11}\) to absolute precision 10, and this polynomial does not have any roots. Hence, one should not interpret the existence of roots of an inexact polynomial to necessarily be roots of any lift of that polynomial to something more precise.

On the other hand, a \verb|Roots| intrinsic for exact polynomials should only return genuine roots of the full-precision polynomial. We can use the inexact \verb|Roots| intrinsic and \verb|IsHenselLiftable| to achieve the desired result:

\begin{lstlisting}
intrinsic Roots(f :: RngUPolElt_FldPadExact) -> []
  for n in 1,2,... do
    // get an approximation to f
    xf := EpochApproximation(f, n)
    // compute the roots of f up to precision
    xroots := Roots(xf)
    // check that the roots are all Hensel liftable
    roots := []
    for xroot in xroots do
      // the roots must be distinct, up to precision,
      // to have a chance of succeeding; if not, go
      // to the next precision in the strategy
      if Multiplicity(xroot) ne 1 then
        continue n
      // see if an approximation to the root is
      // Hensel liftable to a genuine root of f
      ok, root := IsHenselLiftable(f, xroot)
      // if not, then go to the next precision
      if not ok then
        continue n
      // if we get this far, we have a root
      Append(~roots, root)
    // if we get this far, we have a full set of roots
    return roots
  // if we get this far, we have run out of things to try
  error "precision error"
\end{lstlisting}

Note that this can only succeed if all of the roots over the base field are simple, because Hensel's lemma can only detect simple roots. This is the best possible: if \(f\) has a root \(r\) of multiplicity \(m\), then to any precision this is indistinguishable from \(f\) having an irreducible factor of degree \(m\), all of whose roots are very close to \(r\). For example, over \(\QQ_2\), the root \(1\) to multiplicity \(m\) is indistinguishable to high precision from an irreducible factor whose roots are \(1 + 2^{10000} \sqrt[m]{2}\). Hence it is not possible to prove that a polynomial to any finite precision has repeated roots.

\subsection{Hensel's lemma for multivariate root finding}
\label{xp-sec-hensel-multiroot}

We are now interested in solving square systems of multivariate polynomials, namely we wish to find the roots of systems of \(n\) polynomials \(f(x)=(f_1(x),\ldots,f_n(x)) \in K[x]^n\) in \(n\) variables \(x=(x_1,\ldots,x_n)\). A root of such a system is an element \(r \in K^n\) such that \(f(r)=0\).

The following multivariate version of Hensel's lemma is well-known:

\begin{lemma}
\label{xp-lem-hensel-multiroot}
Suppose \(f(x) \in \cO[x]^n\) is a system of \(n\) polynomials in \(n\) variables, \(a \in \cO^n\), \(\val(f(a)) \geq s > 2t = 2\val(\det J(f)(a))\) where \(J(f)_{i,j} = \tfrac{d f_i}{d x_j}\). Then there is a unique \(b \in K^n\) so that \(f(b)=0\) and \(\val(a-b) \geq s-t\). More precisely, defining \(a' = a - f(a) J(f)(a)^{-1}\), then \(\val(\det J(f)(a')) = t\) and \(v(f(a')) \geq 2(s-t)\); therefore iterating \(a \mapsto a'\) then \(a \to b\).
\end{lemma}

We can state a slightly more general version, which says that if we can apply a linear change to the equations, perhaps over an extension, such that Hensel's lemma applies, then Hensel's lemma also applies to the original system over the base field:

\begin{lemma}
\label{xp-lem-hensel-multiroot-general}
Suppose \(f(x) \in K[x]\) is a system of \(n\) polynomials in \(n\) variables, \(a \in K^n\), \(L/K\) a finite extension, \(M,N \in \GL_n(L)\), \(\tilde{a} := M a \in \cO_L^n\), \(\tilde{f} := N f(M^{-1} x) \in \cO_L[x]^n\), \(\val(\tilde{f}(\tilde{a})) \geq s > 2t = \val(\det J(\tilde{f})(\tilde{a}))\). Then \(a\) Hensel lifts to a unique root of \(f\) in \(K\).
\end{lemma}

\begin{proof}
Define \(\tilde{a} = M a\), \(\tilde{a}' = \tilde{a} - \tilde{f}(\tilde{a}) J(\tilde{f})(\tilde{a})^{-1}\). We know that iterating \(\tilde{a} \mapsto \tilde{a}'\) then \(\tilde{a} \to \tilde{b} \in L\) a root of \(\tilde{f}\). By linearity we find \(\tilde{a}' = M a'\) where \(a' = a - f(a) J(f)(a)^{-1}\). We conclude that \(a \to b\) such that \(M b = \tilde{b}\), and since \(a' \in K\), then \(b \in K\) also.
\end{proof}

In the package, we provide an intrinsic \verb|IsHenselLiftable| which takes as input such a system \(f\) and a near-root \(a\) and returns true if Hensel's lemma is applicable. If so, it also returns the Hensel-lifted root itself. It also optionally accepts two vectors \(\mu,\nu \in \QQ^n\) which define the diagonal matrices \(M\) and \(N\) with diagonal entries \(\pi^\mu\) and \(\pi^\nu\), and uses the more general version of Hensel's lemma. This allows us to implicitly rescale the equations and variables, so that the inputs need not be integral.

It should be possible to determine whether there exists any such \(\mu\) and \(\nu\) so that Hensel's lemma is applicable, and therefore recover a completely general and parameterless version of multivariate \verb|IsHenselLiftable| in analogue with the univariate case. The theory for this has not been completely worked out yet.

\begin{remark}
An algorithm to actually compute the roots or factors of such a square system is work in progress.
\end{remark}

\subsection{Hensel's lemma for univariate factorization}
\label{xp-sec-hensel-factorization}

Suppose \(f(x) \in K[x]\) is a monic univariate polynomial of degree \(n=n_1+n_2\). Consider the problem of finding a factorization \(f(x) = g(x)h(x)\) where \(\deg(g)=n_1\), \(\deg(h)=n_2\) and \(g\) and \(h\) are monic. By treating the \(n\) coefficients of \(1,x,\ldots,x^{n-1}\) in \(f(x)-g(x)h(x)\) as multivariate polynomials in the \(n_1\) coefficients of \(g\) and the \(n_2\) coefficients of \(h\), we have a system of \(n\) multivariate polynomials in \(n\) variables to solve.

We conclude that there is a version of Hensel's lemma applicable to this situation, provided that a given near-factorization \(f(x) \approx g(x)h(x)\) is sufficiently accurate. How accurate this needs to be is controlled by the determinant of the Jacobian matrix \(J\) in Hensel's lemma. In this case, the first \(n_1\) rows of \(J\) correspond to \(\tfrac{d(f-gh)}{dg_i}=x^i h(x)\), and the next \(n_2\) rows correspond to \(\tfrac{d(f-gh)}{dh_i}=x^i g(x)\), with the columns being the coefficients of these polynomials. This is precisely the matrix defining the resultant, and so we conclude that \(\det(J) = \Res(g, h)\).

For example, we get the following version of Hensel's lemma for factorization, although more general versions analogous to those in previous sections are also possible.

\begin{lemma}
Suppose \(f(x),g(x),h(x) \in \cO[x]\) are monic of degrees \(n=n_1+n_2,n_1,n_2\) such that \(\val(f-gh) \geq s > 2t = 2\Res(g,h)\). Then \(g,h\) Hensel-lift uniquely to a factorization of \(f\).
\end{lemma}

Suppose we are given \(f(x)\) and \(g(x)\) but not \(h(x)\) and want to determine if \(g(x)\) is Hensel liftable to a factor of \(f(x)\). It seems natural to define \(h := f \xdiv g\) and apply Hensel's lemma to this. The following lemma shows that this is indeed the best choice for \(h\):

\begin{lemma}
If \(f(x),g(x) \in K[x]\) have degrees \(n\) and \(n_1 \leq n\) and \(g\) is monic, then among polynomials \(h(x) \in K[x]\) of degree \(n_2=n-n_1\), \(\val(f-gh)\) is maximized by \(h = f \xdiv g\).
\end{lemma}

\begin{proof}
By definition, \(f - g (f \xdiv g) = f \bmod g =: h_0\). Consider arbitrary \(h = f \xdiv g + d\), then \(f-gh = f - g (f \xdiv g) - g d = h_0 - g d\). Define \(B = \val(h_0) + 1\) and suppose there exists \(d\) so that \(\val(h_0)-gd \geq B\). In particular \(d \neq 0\). Fix \(d\) of smallest degree, and let \(m\) be this degree. Then the \((m+n_1)\)th coefficient of \(f-gh\) is \(-d_m\) and so \(\val(d_m) \geq B\). Define \(d' = d - d_m x^m\), then \(\val(h_0)-gd' \geq B\) and \(\deg d' < \deg d\), a contradiction.
\end{proof}

The package provides an intrinsic \verb|IsHenselLiftable| which takes as input two polynomials \(f\) and \(g\) and returns true if \(g\) is Hensel-liftable to a factor of \(f\). If so, it also returns the factor itself. In analogue with the multivariate version of \verb|IsHenselLiftable|, this intrinsic takes parameters which implicitly re-scale the polynomials and the variable \(x\) before applying Hensel's lemma.

\subsection{Univariate factorization by Newton polygon}

An easy application of Hensel's lemma for univariate factorizations is to factor a polynomial according to its Newton polygon.

Recall that the slopes of faces of the Newton polygon of a polynomial \(f(x)\) correspond to valuations of roots of \(f(x)\), with the width of the face corresponding to the number of roots with this valuation. If two roots of \(f(x)\) come from the same irreducible factor, then they are Galois conjugate and so have the same valuation; we conclude that each face of the Newton polygon corresponds to a factor of \(f\) whose degree is the width of the face.

In fact, we can prove this fact directly using a version of Hensel's lemma for factoring, seen in the previous section: it is not hard to see that with a suitable choice of rescaling on \(f\) and \(x\) that we may choose \(g\) so that Hensel's lemma is applicable. Specifically, we rescale so that the selected face of the Newton polygon of \(f\) becomes horizontal and incident with the x-axis, and take for \(g\) the polynomial formed from the coefficients of \(f\) corresponding to the face.

The package provides a routine \verb|NewtonPolygonFactorization| which takes as input a univariate polynomial \(f\) and returns its factorization according to its Newton polygon. It is implemented essentially by first computing the \verb|NewtonPolygon| of \(f\), and then for each face constructing a suitable \(g\) and calling \verb|IsHenselLiftable| to produce a factor.

\subsection{Univariate factorization into irreducibles I}
\label{xp-sec-factorization}

We also provide an intrinsic \verb|Factorization| which returns the full factorization of a polynomial \(f(x)\) into irreducible factors.

It is implemented in a very similar fashion to \verb|Roots| as described in \S\ref{xp-sec-roots}: it calls Magma's builtin \verb|Factorization| routine on an approximation to \(f(x)\), and then checks if each factor returned is Hensel liftable using \verb|IsHenselLiftable|.

\subsection{Univariate root finding and factorization into irreducibles II}
\label{xp-sec-factorization-new}

Our \verb|Roots| and \verb|Factorization| intrinsics actually have a parameter \verb|Alg| to select between two different algorithms. We have already described \verb|Alg:="Builtin"| (\S\ref{xp-sec-roots}, \S\ref{xp-sec-factorization}) which is a wrapper around the builtin intrinsics for inexact \(p\)-adics.

With the parameter \verb|Alg:="OM"|, which is now the default, we use our own implementation of an ``OM algorithm'' for computing ``Okutsu invariants'' of the input polynomial, which identifies its irreducible factors and some properties of the extensions they define. From these, we can use ``single factor lifting'' to generate arbitrarily precise approximations to the factors. The algorithm is essentially that described in \cite[Ch. VI]{SinclairTh}.

\begin{remark}
Although not usually presented as such, ``single factor lifting'' is nothing but Hensel's lemma in disguise. Recall in \ref{xp-sec-hensel-factorization} that we expressed factoring \(f(x)=g(x)h(x)\) as a multivariate system of equations whose coefficients are the \(n_1=\deg(g)\) coefficients of \(g\) and the \(n_2=\deg(h)\) coefficients of \(h\).

We can instead write \(g(x) = x^{n_1} + \sum_{i<n_1} g'_i X_{g,i}(x)\) and \(h(x) = x^{n_2} + \sum_{i<n_2} h'_i X_{h,i}(x)\) where \(X_{*,i}(x) \in K[x]\) are fixed monic polynomials of degree \(i\), and instead consider \(f(x)=g(x)h(x)\) as a system of equations in the variables \(g'_i\) and \(h'_i\). Essentially, we have chosen bases for the vector spaces of monic polynomials of degrees \(n_1\) and \(n_2\) different from the usual \(1,x,x^2,\ldots\). This is a linear change of variables of the sort considered in Lemma \ref{xp-lem-hensel-multiroot-general}.

The OM algorithm builds up such a basis for each factor, and the point in the algorithm at which an irreducible factor is identified is precisely the point at which Hensel's lemma, in terms of this basis, can be invoked.
\end{remark}

\begin{remark}
The same algorithm is also made available as an intrinsic \verb|ExactpAdics_Factorization| which can take an inexact \(p\)-adic polynomial. This can be used independently of the package.
\end{remark}

\bibliography{refs}

\end{document}